\documentclass[11.5pt]{amsart}

\usepackage{amsmath,amsthm,amssymb,latexsym,amscd}
\usepackage{setspace}
\usepackage[dvips]{graphicx}
\usepackage{longtable} %% if you have a table spread more than a page.
\usepackage{float}
\usepackage{subcaption}
\usepackage{pstricks}  %% if you draw pictures in pstricks.
\usepackage{pst-node}  %% if you draw pictures in pstricks.
\usepackage{scalefnt}  %% if you want to scale the fonts (make it bigger or smaller).
\usepackage{enumitem}
\usepackage{mdwlist}
\usepackage{tikz}
\usepackage{bm}
\usepackage{thmtools}
\usepackage{thm-restate}
\usepackage{cleveref}
\usepackage[mathscr]{euscript}
\usepackage[parfill]{parskip}
\usepackage[toc,page]{appendix}

\usepackage{algorithmf}
\usepackage{algpseudocode}
\usepackage{algorithmicx}
\usepackage{algorithm}

\algnewcommand\And{\textbf{and}}

\DeclareMathAlphabet{\mathpzc}{OT1}{pzc}{m}{it}

\usetikzlibrary{arrows,scopes, shapes.geometric, calc, shapes, decorations.markings}
\tikzstyle{vertex}=[circle, draw, fill, inner sep=0pt, minimum size=2pt]

\usepackage[margin=1.25in]{geometry}

\theoremstyle{plain}
\newtheorem{theorem}{Theorem}[section]
\newtheorem{lemma}[theorem]{Lemma}
\newtheorem{corollary}[theorem]{Corollary}

\newtheorem{conjecture}[theorem]{Conjecture}

\theoremstyle{definition}
\newtheorem{definition}[theorem]{Definition}

\newtheorem{observation}[theorem]{Observation}

\theoremstyle{remark}

\numberwithin{equation}{theorem}

\title{A Highly Symmetric Hamilton Decomposition for Hypercubes}
\author[F. Bouya]{Farid Bouya}
\address{Department of Mathematics, Louisiana State University}
\email{fbouya1@lsu.edu}
\author[E. S. Mahmoodian]{Ebadollah S. Mahmoodian}
\address{Department of Mathematics, Sharif University of Technology}
\email{emahmood@sharif.ir}
\author[M. Shokrian]{Modjtaba Shokrian Zini}
\address{Department of Mathematics, University of California, Santa Barbara}
\email{shokrian@math.ucsb.edu}
\author[M. Tefagh]{Mojtaba Tefagh}
\address{Department of Electrical Engineering, Stanford University}
\email{mtefagh@stanford.edu}

\date{\today}

\def\i4c{internally $4$-connected}

\def\int{\mathrm{int}}

\def\le{\leqslant}

\def\i4c{internally $4$-connected}

\def\'{$'$}

\def\i4c{internally 4-connected}

\begin{document}

\begin{abstract}
A Hamilton decomposition of a graph is a partitioning of its edge set into disjoint spanning cycles.
The existence of such decompositions is known for all hypercubes of even dimension $2n$.
We give a decomposition for the case $n = 2^a3^b$ that is highly symmetric in the sense that every cycle can be derived from every other cycle just by permuting the axes.
We conjecture that a similar decomposition exists for every $n$.

{\it{ \bf keywords}}: Hamilton decomposition, Hypercubes
\end{abstract}

\maketitle

\section{Introduction}

%\cite{aub}
%\cite{bass}

%\cite{Okuda}
%\cite{ram}
%\cite{tow}

%A \emph{Hamilton cycle} in a graph $G$ is a cycle that visits all the vertices.
%A \emph{Hamilton decomposition} of $Q_{2n}$ is partitioning the edges of $Q_{2n}$ into $n$ edge-disjoint Hamilton cycles.
Hypercubes are widely used in computer architectures in areas like parallel computing~\cite{ost}, multiprocessor systems~\cite{das}, processor allocation~\cite{rai}, and fault-tolerant computing~\cite{abd}.
Hamilton decomposition (H.D.) of hypercubes is of central importance in the aforementioned areas.

%Another totally different field of application is the use of Gray codes in DNA computing and large-scale motif discovery.
In 1954, Ringel showed that the hypercube $Q_n$ is Hamilton decomposable whenever $n$ is a power of two and posed the problem of whether a similar decomposition exists for all even~$n$~\cite{ram}.
In 1982, Aubert and Schneider showed that every $Q_{2n}$ admits a Hamilton decomposition~\cite{ram}.
%In 1973 K\"{o}tzig  conjectured if $G_1,G_2, \ldots ,G_n$ are all decomposible into $m$ Hamilton cycles then the cartesian product of $G_i's$, $G_1 \Box G_2 \Box \cdots \Box G_n$ can be decomposed into $mn$ cycles.
%Also, applying this later conjecture to $G_i=C_4$ implies the Hamilton decomposition (H.D.) of $Q_{2n}$.
%In 1978, Foregger proved K\"{o}tzig's conjecture for $n=2^a3^b$.
%A more general theorem, solving K\"{o}tzig's Conjecture, was proved by R. Stong in 1991 by a nonconstructive proof, and thus the existence of Hamilton decomposition for hypercubes is known.
Many different algorithms and methods have been used to find explicit Hamilton decompositions for $Q_{2n}$.
%Bass and Sudborough \cite{bass} gave two recursive algorithms, one going from $Q_{2n}$ to $Q_{4n}$, and the other going from $Q_{2n}$ to $Q_{2n+2}$.
Our work is inspired by two such methods.
Okuda and Song~\cite{Okuda} gave a direct approach for finding Hamilton decompositions for $Q_{2n}$ with $n \le 4$.
Mollard and Ramras~\cite{ram} gave a fast and efficient method of generating and storing Hamilton decompositions when $n$ is a power of two by constructing one special cycle and permuting the axes to obtain the other cycles.
We use Okuda and Song's method to continue the work of Mollard and Ramras and extend it to all $n$ of the form $2^a 3^b$, which is the main result of this paper, stated formally as Corollary~\ref{c5}.
In Appendix~\ref{ap2}, we present Algorithms~\ref{a4} and~\ref{a5} that efficiently construct such decompositions.
We conjecture that a similar decomposition exists for every~$n$.

%There may be many ways to do a Hamilton decomposition on $Q_{2n}$.  We aim to do this task in such a way that any of the cycles can be derived from any other cycle by just permuting the axes of $Q_{2n}$.

\section{Notations}

The \emph{hypercube of dimension $n$}, denoted by $Q_n$, is the graph whose vertices are the $2^n$ binary strings of length $n$ and two vertices are adjacent if and only if their corresponding strings differ in exactly one bit.
The \emph{Cartesian product} of two graphs $G$ and $H$, denoted by $G \Box H$, has vertex set
\begin{align*}
V \left( G \Box H \right) = \left\{ (u,v) \vert u \in V(G) \text{ and } v \in V(H) \right\}
\end{align*}
and two vertices $(u,v)$ and $(u',v')$ are connected if and only if
\begin{itemize}
    \item $u = u'$ and $vv' \in E(H)$, or
    \item $v = v'$ and $uu' \in E(G)$.
\end{itemize}
%\normalfont

Using this definition, it is not hard to see that
\begin{align*}
& Q_{m+n} = Q_m \Box Q_n, \\
& Q_n = \underbrace{ K_2 \Box K_2 \Box \cdots \Box K_2}_{n \text{ times}}, \\
\end{align*}
and
\begin{align}
\label{f1}
& Q_{2n} = \underbrace{C_4 \Box C_4 \Box \cdots \Box C_4}_{n \text{ times}}.
\end{align}
%\itshape

As in~\cite{tow}, we use Equation (\ref{f1}) to make another coordinate system for the vertices of $Q_{2n}$: \\
Each vertex is assigned a quaternary string  $q_1 q_2 \ldots q_n$ of length $n$, where $q_i \in \{ 0,1,2,3 \}$.
There is an edge between two vertices if and only if their labels differ in exactly one position, and in that position, their difference is either $1$ or $-1$ modulo $4$.
We wish to consider directed cycles, so we assign directions to edges of $Q_{2n}$ as follows:
A \emph{dimension-$k$ edge} in $Q_{2n}$ is an edge that connects two vertices whose quaternary labels differ in the {\it$k$th} digit.
If a dimension-$k$ edge is directed in the \emph{positive direction}, that is, it is directed from $(x_1,x_2, \ldots ,x_{k-1}, x_k, x_{k+1} , \ldots, x_n) $ towards $(x_1,x_2, \ldots ,x_{k-1}, x_k +1 \pmod{4}, x_{k+1} , \ldots, x_n) $, we show it by $k$, and if it is directed in the opposite direction, we show it by $\overline{k}$.
A \emph{Hamilton cycle} in a graph is a cycle visiting all the vertices.
A \emph{Hamilton decomposition} of $Q_{2n}$ is a partitioning of its edge set into $n$ disjoint Hamilton cycles.
%All operations are done modulo 4.
%We also want to be able to show a directed cycle.
We use the notation given in~\cite{Okuda} to show directed cycles:
We start from the initial vertex, and simply move in the positive direction of $C$, writing down the dimension and the direction of the edges we pass.
For example, the cycle given in Figure~\ref{p1}, with the origin (top left vertex) as its initial vertex, is shown by $2\overline{11}22\overline{1112}11\overline{22}111$.\\
\begin{figure}[!hbtp]
\centering

\begin{tikzpicture}
%\tikzstyle{vertex}=[circle, draw, inner sep=0pt, minimum size=6pt]
    \foreach \i in {0,...,3} {
        \draw [dotted ,gray] (\i,0) -- (\i,3) % node [below] at (\i,0) {$\i$}
        ;
            \foreach \j in {0,...,3} {
        \node[vertex] at (\i,\j) {};
}
            
    }
    \foreach \i in {0,...,3} {
        \draw [dotted,gray] (0,\i) -- (3,\i) % node [left] at (0,\i) {$\i$}
        ;
    }
    
\draw[->,ultra thick] (-1,4)--(0,4) node[right]{$1$};
\draw[->,ultra thick] (-1,4)--(-1,3) node[below]{$2$};

\begin{scope}[very thick,decoration={
    markings,
    mark=at position 0.5 with {\arrow{>}}}
    ] 
    \draw[postaction={decorate}] (0,3)--(0,2);
    \draw[postaction={decorate}] (3,2)--(2,2);
    \draw[postaction={decorate}] (2,2)--(2,0);
    \draw[postaction={decorate}] (2,0)--(0,0);
    \draw[postaction={decorate}] (3,0)--(3,1);
    \draw[postaction={decorate}] (0,1)--(1,1);
    \draw[postaction={decorate}] (1,1)--(1,3);
    \draw[postaction={decorate}] (1,3)--(3,3);
    \draw[postaction={decorate}] (3, 3) .. controls (1.5,3.5) .. (0, 3);
    \draw[postaction={decorate}] (0, 2) .. controls (1.5,2.5) .. (3, 2);
    \draw[postaction={decorate}] (0, 0) .. controls (1.5,.5) .. (3, 0);
    \draw[postaction={decorate}] (3, 1) .. controls (1.5,1.5) .. (0, 1);
\end{scope}

        \node[vertex, fill, minimum size = 4pt] at (0,3) {};

%\draw [step=1.0,blue, very thick] (0.5,0.5) grid (5.5,4.5);
%\draw [very thick, brown, step=1.0cm,xshift=-0.5cm, yshift=-0.5cm] (0.5,0.5) grid +(5.5,4.5);
\end{tikzpicture}

\caption{ \footnotesize The cycle in $Q_4$ matching the code $2\overline{11}22\overline{1112}11\overline{22}111$.}\label{p1}
\end{figure}
In this paper, we only deal with Hamilton cycles, so the initial vertex is always taken to be the origin, that is, the vertex $\mathbf{0} =  (0,0, \ldots ,0)$.

\begin{definition}
Define $G_{n,k}$ to be the graph $ \underbrace{C_{4^n} \Box C_{4^n} \Box \cdots \Box C_{4^n}}_{k \text{ times}}$.
\end{definition}
Note that $Q_{2k} \cong G_{1,k}$.
We show directed edges and cycles in $G_{n,k}$ the same way we show them in $Q_{2n}$.
The only difference is that coordinates in $G_{n,k}$ are calculated modulo $4^n$, whereas they are calculated modulo 4 in $Q_{2n}$.
We are especially interested in the cases $k = 2$ and $k = 3$, so we recognize that these cases require special treatment.
Since $G_{n,2}$ is the Cartesian product of two cycles, we think of $G_{n,2}$ as a 2-dimensional cyclic grid.
Every vertex of $C_{4^n} \Box C_{4^n}$ has coordinates $(u,v)$, where $u$ is in the first copy of $C_{4^n}$ and $v$ is in the second copy.
We think of this coordinate $(u,v)$ in two ways:
\begin{enumerate}
\item The vertices $u$ and $v$ are elements of $Q_{2n}$, and thus quaternary strings of length $n$.
Therefore, $(u,v)$ is a quaternary string of length $2n$.
\item Fixing some order in $Q_{2n}$, we assign the integers $0$ to $4^n-1$ to its vertices.
Thus, every vertex in $Q_{4n}$ has integral coordinates $(u,v)$ where $0 \le u,v \le 4^n-1$.
\end{enumerate}

In order to derive Hamilton decompositions for larger hypercubes from smaller hypercubes, we study the graphs $G_{n,2}$ and $G_{n,3}$ in more detail.

\section{The 2-Dimensional Case}
We start by finding an H.D. for $G_{n,2}$.
We will see how an H.D. for $G_{n,2}$ and an H.D. for $Q_{2n}$ can be combined to give an H.D. for $Q_{4n}$.

\subsection{An H.D. for $G_{n,2}$}\hspace*{\fill} \label{sec31}

There is an H.D. for the graph $C_m \Box C_m$ in
\begin{align*}
& H_1 = \underbrace{\underbrace{1 1 \cdots 1}_{m-1 \text{ times}} 2 \underbrace{1 1 \cdots 1}_{m-1 \text{ times}} 2 \cdots \underbrace{1 1 \cdots 1}_{m-1 \text{ times}} 2}_{m \text{ times}} \hspace{10pt} , \hspace{10pt}
H_2 = \underbrace{\underbrace{2 2 \cdots 2}_{m-1 \text{ times}} 1 \underbrace{2 2 \cdots 2}_{m-1 \text{ times}} 1 \cdots \underbrace{2 2 \cdots 2}_{m-1 \text{ times}} 1}_{m \text{ times}}
\end{align*}
Since $G_{n,2} = C_{4^n} \Box C_{4^n}$, we have the same type of H.D. for $G_{n,2}$:
\begin{align}
\label{f2}
& H_1 = \underbrace{\underbrace{1 1 \cdots 1}_{4^n-1 \text{ times}} 2 \underbrace{1 1 \cdots 1}_{4^n-1 \text{ times}} 2 \cdots \underbrace{1 1 \cdots 1}_{4^n-1 \text{ times}} 2}_{4^n \text{ times}} \hspace{10pt} , \hspace{10pt}
H_2 = \underbrace{\underbrace{2 2 \cdots 2}_{4^n-1 \text{ times}} 1 \underbrace{2 2 \cdots 2}_{4^n-1 \text{ times}} 1 \cdots \underbrace{2 2 \cdots 2}_{4^n-1 \text{ times}} 1}_{4^n \text{ times}}
\end{align}

\subsection{Deriving an H.D. for $Q_{4n}$ From an H.D. for $Q_{2n}$}\hspace*{\fill}

Noting that $Q_{2n}$ has order $4^n$ and $Q_{4n} = Q_{2n} \Box Q_{2n}$, we propose the following:
\begin{definition}
Let $E$ be a directed Hamilton cycle in $Q_{2n}$.
A \emph{2-dimensional seating} of $Q_{4n}$ onto $G_{n,2}$ via $E$, is a representation of the vertices of $Q_{4n}$ by assigning them integral coordinates as follows:
\begin{enumerate}
    \item Consider $E$ and its positive direction.
    Take $\bm{0}$ as the initial vertex.
    Assign $0$ to $\bm{0}$, assign $1$ to the next vertex in $E$, and continue until $4^n-1$ is assigned to the last vertex of $E$.
    \item Induce the order of $E$ onto $Q_{2n}$, so that each vertex has the same order in either graph. \label{d22}
    \item Using the coordinates in (\ref{d22}), assign coordinates to every member of $Q_{4n} = Q_{2n} \Box Q_{2n}$.
    Put the vertices on the 2-dimensional grid using their coordinates.
\end{enumerate}
Using the natural order of $E$, we have mapped the vertices of $Q_{4n}$ onto $G_{n,2}$ and recognized $Q_{4n}$ as a supergraph of $G_{n,2}$.
Any subgraph of $G_{n,2}$, therefore, is also a subgraph of $Q_{4n}$.
In particular, if $H$ is a directed Hamilton cycle in $G_{n,2}$, the \emph{2-dimensional directed Hamilton cycle derived from $E$ and $H$}, denoted by $f(E,H)$, is a Hamilton cycle in $Q_{4n}$ and is defined in the natural way:
\begin{enumerate}
\item 2-dimensionally seat $Q_{4n}$ onto $G_{n,2}$ via $E$.
\item $Q_{4n}$ has $2n$ axes $1,2, \ldots , 2n$, while $G_{n,2}$ has an $x$-axis and a $y$-axis.
The axes $1,2, \ldots , n$ are in direction $x$ and the axes $n+1,n+2, \ldots , 2n$ are in direction $y$.
\item $f(E,H)$ has the same edges in the supergraph $Q_{2n} \Box Q_{2n}$ as $H$ has in the subgraph $G_{n,2}$.
\end{enumerate}

\end{definition}

The following lemma is very useful.
\begin{lemma}
\label{l1}
Let $H_1$ and $H_2$ be two disjoint Hamilton cycles in  $G_{n,2}$ (which form an H.D.) and $E_1$ and $E_2$ be two disjoint Hamilton cycles in $Q_{2n}$. Then the four Hamilton cycles $F_1 = f(E_1 , H_1)$, $F_2 = f(E_1,H_2)$, $F_3 = f(E_2 , H_1)$, and $F_4 = f(E_2 , H_2)$ in $Q_{4n}$ are pairwise disjoint.
\end{lemma}
\begin{proof}
It suffices to show that $F_1 = f(E_1, H_1)$ is disjoint from the other three cycles $F_2$, $F_3$, and $F_4$.
To achieve this, we 2-dimensionally seat $Q_{4n}$ onto $G_{n,2}$ via $E_1$.
This enables us to see that $F_1$ and $F_2$ have all their edges on the grid, whereas $F_3$ and $F_4$ have all their edges off the grid.
This means that $F_1$ is disjoint from $F_3$ and from $F_4$.
Furthermore, $F_1$ and $F_2$ represent $H_1$ and $H_2$, respectively, and $H_1$ and $H_2$ are disjoint, so $F_1$ and $F_2$ must be disjoint as well.
\end{proof}
This provides us with a recursive tool to construct Hamilton decompositions.
\begin{corollary}
\label{c1}
If $\{ H_1 , H_2 \} $ is an H.D. for $G_{n,2}$ and $\{ E_1 , E_2 , \ldots , E_{n}  \} $ is an H.D. for $Q_{2n}$, then the family $\{ f(E_i , H_j) \mid 1 \le i \le n, 1 \le j \le 2 \}$ is an H.D. for $Q_{4n}$.
The new Hamilton cycles are named $F_1$, $F_2$, \ldots, $F_{2n}$ via $F_{j} = f(E_j,H_1)$ and $F_{j+n} = f(E_j,H_2)$ for $1 \le j \le n$.
\end{corollary}

\subsection{2-Dimensional Algorithm}\hspace*{\fill}

Using the definition of $f(E,H)$, it is not difficult to devise an algorithm for computing an H.D. for $Q_{4n}$.
Algorithm~\ref{a1}, given in Appendix~\ref{ap2}, takes an H.D. for $G_{n,2}$ and an $H.D.$ for $Q_{2n}$ as inputs, and outputs an H.D. for $Q_{4n}$.
%It does not produce highly sy The second one, given in \cite{MR1989980}, uses the H.D. given in (\ref{f2}) to do the job.

\section{The 3-Dimensional Case}
Just like in the 2-dimensional case, finding an H.D. for the graph $G_{n,3}$ is essential for transitioning from $Q_{2n}$ to $Q_{6n}$.
An H.D. for $Q_{2n}$ can be combined with an H.D. for $G_{n,3}$ to give an H.D. for $Q_{6n}$.

%%%%%%%%%%%%%%%%%%%%%%%%%%%%%%%%%%%%%%%%%%%%%%%%%%%%%%%%%%%%555

\subsection{An H.D. for $G_{n,3}$}\hspace*{\fill} \label{sec41}

Compared to the 2-dimensional case, finding an H.D. for $G_{n,3}$ is not easy. Motivated by~\cite{Okuda} and~\cite{tow}, we decompose the graph into three 2-factors, and then try to merge the components until we have three Hamilton cycles.
\begin{lemma}
\label{l2}
The graph $G_{n,3}$ with the partitioning given below decomposes into $3 \times 4^{n}$ copies of the directed cycle with $4^{2n}$ edges:

If $e$ is in direction 1 and is between $(x,y,z)$ and $(x+1,y,z)$, we direct $e$ from $(x,y,z)$ to $(x+1,y,z)$ and
\begin{align*}
& e \in Z \qquad \text{if } x+y+z = -1 \pmod{4^n}, \\
& e \in X \qquad \text{otherwise.}
\end{align*}
If $e$ is in direction 2 and is between $(x,y,z)$ and $(x,y+1,z)$, we direct $e$ from $(x,y,z)$ to $(x,y+1,z)$ and
\begin{align*}
& e \in X \qquad \text{if } x+y+z = -1 \pmod{4^n}, \\
& e \in Y \qquad \text{otherwise.}
\end{align*}
If $e$ is in direction 3 and is between $(x,y,z)$ and $(x,y,z+1)$, we direct $e$ from $(x,y,z)$ to $(x,y,z+1)$ and
\begin{align*}
& e \in Y \qquad \text{if } x+y+z = -1 \pmod{4^n}, \\
& e \in Z \qquad \text{otherwise.}
\end{align*}
\end{lemma}

We have demonstrated the case $n=1$ in Appendix~\ref{ap1}.

\begin{proof}
Choosing an arbitrary vertex $v$ and moving along the edges of $X$, we can see that $v$ belongs to a unique cycle of length $4^{2n}$ that is in $X$. Similarly, it belongs to a unique cycle of length $4^{2n}$ in $Y$ and another one in $Z$.
There are $4^{3n}$ vertices in total, so there are $4^n$ cycles in each of $X$, $Y$, and $Z$, for a total of $3 \times 4^n$ cycles.
\end{proof}

We wish to merge these cycles together and end up with just three, so that we have an H.D. for $G_{n,3}$. To this end, we introduce two cubes and a merge operation. These cubes and the merge operation were first introduced in~\cite{tow} and later in~\cite{Okuda} to build an H.D. for $Q_6$. We use them to construct an H.D. for every $G_{n,3}$

\begin{definition}
Let $c_X$, $c_Y$, and $c_Z$, denote the number of (current) connected components of $X$, $Y$, and $Z$, respectively.
%We define two special colorings of the 3-dimensional cube in figure \ref{p2}.
The \emph{type-I cube} and the \emph{type-II cube} are given in Figure~\ref{p2}.
The top left vertex is the \emph{origin of the cube}, that is, the vertex $(x,y,z)$ such that any other vertex $(x',y',z')$ of the cube satisfies $0 \le x-x', y-y', z-z' \le 1 \pmod{4^n}$.

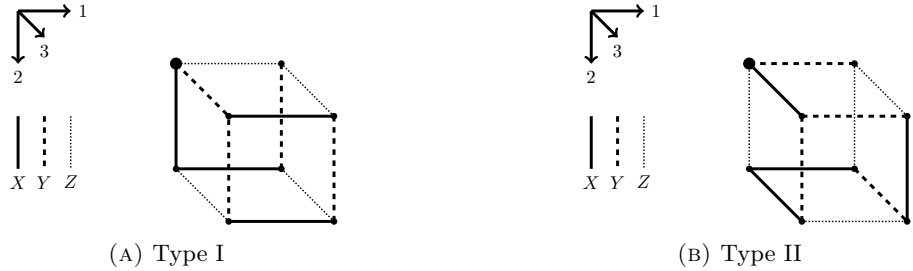
\begin{figure}[H]
\centering
\begin{subfigure}{.5\textwidth}
  \centering

\scalebox{.7}{
\begin{tikzpicture}

\foreach \i in {0,1}{
	\foreach \j in {0,1}{
		\foreach \k in {0,1}{
		        \node[vertex, fill, minimum size = 3pt] at (0+ 2 * \i + \k, 3 - 2 * \j - \k) {};
      		        }
	}
}
\node[vertex, fill, minimum size = 6pt] at (0, 3) {};

\draw[->,ultra thick] (-3,4)--(-2,4) node[right]{$1$};
\draw[->,ultra thick] (-3,4)--(-3,3) node[below]{$2$};
\draw[->,ultra thick] (-3,4)--(-2.5,3.5) node[below]{$3$};

\draw[solid,ultra thick] (-3,2)--(-3,1) node[below]{$X$};
\draw[dashed,ultra thick] (-2.5,2)--(-2.5,1) node[below]{$Y$};
\draw[densely dotted,thick] (-2,2)--(-2,1) node[below]{$Z$};

\draw[solid,ultra thick] (0,3)--(0,1)--(2,1);
\draw[solid,ultra thick] (1,0)--(3,0);
\draw[solid,ultra thick] (1,2)--(3,2);

\draw[dashed,ultra thick] (0,3)--(1,2)--(1,0);
\draw[dashed,ultra thick] (2,3)--(2,1);
\draw[dashed,ultra thick] (3,0)--(3,2);

\draw[densely dotted, thick] (3,2)--(2,3)--(0,3);
\draw[densely dotted, thick] (0,1)--(1,0);
\draw[densely dotted, thick] (2,1)--(3,0);
\end{tikzpicture}
}

  \subcaption{Type I \label{x2a}}
\end{subfigure}%
\begin{subfigure}{.5\textwidth}
  \centering

\scalebox{.7}{
\begin{tikzpicture}
\foreach \i in {0,1}{
	\foreach \j in {0,1}{
		\foreach \k in {0,1}{
		        \node[vertex, fill, minimum size = 3pt] at (0+ 2 * \i + \k, 3 - 2 * \j - \k) {};
      		        }
	}
}
\node[vertex, fill, minimum size = 6pt] at (0, 3) {};

\draw[->,ultra thick] (-3,4)--(-2,4) node[right]{$1$};
\draw[->,ultra thick] (-3,4)--(-3,3) node[below]{$2$};
\draw[->,ultra thick] (-3,4)--(-2.5,3.5) node[below]{$3$};

\draw[solid,ultra thick] (-3,2)--(-3,1) node[below]{$X$};
\draw[dashed,ultra thick] (-2.5,2)--(-2.5,1) node[below]{$Y$};
\draw[densely dotted,thick] (-2,2)--(-2,1) node[below]{$Z$};

\draw[densely dotted, thick] (0,3)--(0,1);
\draw[solid,ultra thick] (0,1)--(2,1);
\draw[densely dotted, thick] (1,0)--(3,0);
\draw[dashed,ultra thick] (1,2)--(3,2);

\draw[solid,ultra thick] (0,3)--(1,2);
\draw[dashed,ultra thick](1,2)--(1,0);
\draw[densely dotted, thick] (2,3)--(2,1);
\draw[solid,ultra thick] (3,0)--(3,2);

\draw[densely dotted, thick] (3,2)--(2,3);
\draw[dashed,ultra thick] (2,3)--(0,3);
\draw[solid,ultra thick] (0,1)--(1,0);
\draw[dashed,ultra thick] (2,1)--(3,0);
\end{tikzpicture}
}

  \subcaption{Type II \label{x2b}}
\end{subfigure}

\caption{ \footnotesize The special cube type-I (\ref{x2a}) and the special cube type-II (\ref{x2b}).}\label{p2}

\end{figure}

By \emph{merging} a type-I cube we replace it with a type-II cube.
Note that the vertices maintain their $X$-, $Y$-, and $Z$-degrees during the merge operation.

\end{definition}

The aim of the merge operation is to reduce each of $c_X$, $c_Y$, and $c_Z$ by 1.
Before starting to merge, we need to make sure that we have enough type-I cubes and that this three-way switch in colors indeed merges six cycles into three.
We make a couple of observations.

\begin{observation} \label{obs1}
Consider $G_{n,3}$ and decompose it with the method described in Lemma~\ref{l2}.
Then every vertex $(x,y,z)$ with $x+y+z = -1 \pmod{4^n}$ is the origin of a type-I cube.
\end{observation}

\begin{observation} \label{obs2}
Figure~\ref{p3} shows that, a single merge operation, done on the decomposition obtained from Lemma~\ref{l2}, indeed merges six cycles, two in each of $X$, $Y$, and $Z$, into three cycles, one in each of $X$, $Y$, and $Z$.
\end{observation}

\begin{figure}[H]
\centering
\begin{subfigure}{\textwidth}
  \centering

\begin{tikzpicture}
\foreach \i in {0,2}{
	\foreach \j in {1,3}{
		\foreach \k in {0,1}{
		        \node[vertex, fill, minimum size = 3pt] at (\i + \k, \j - \k) {};
      		        }
	}
}
\foreach \i in {8,10}{
	\foreach \j in {1,3}{
		\foreach \k in {0,1}{
		        \node[vertex, fill, minimum size = 3pt] at (\i + \k, \j - \k) {};
      		        }
	}
}
\node[vertex, fill, minimum size = 6pt] at (0, 3) {};
\node[vertex, fill, minimum size = 6pt] at (8, 3) {};

\node[vertex, fill, minimum size = 6pt] at (0, 3) {};    
\draw[->,ultra thick] (-3,4)--(-2,4) node[right]{$1$};
\draw[->,ultra thick] (-3,4)--(-3,3) node[below]{$2$};
\draw[->,ultra thick] (-3,4)--(-2.5,3.5) node[below]{$3$};

\draw[solid,ultra thick] (-3,2)--(-3,1) node[below]{$X$};
\draw[dashed,ultra thin] (-2.5,2)--(-2.5,1) node[below]{$Y$};
\draw[densely dotted,thin] (-2,2)--(-2,1) node[below]{$Z$};

\draw[solid,ultra thick] (0,3)--(0,1)--(2,1);
\draw[solid,ultra thick] (1,0)--(3,0);
\draw[solid,ultra thick] (1,2)--(3,2);

\draw[dashed,ultra thin] (0,3)--(1,2)--(1,0);
\draw[dashed,ultra thin] (2,3)--(2,1);
\draw[dashed,ultra thin] (3,0)--(3,2);

\draw[densely dotted, thin] (3,2)--(2,3)--(0,3);
\draw[densely dotted, thin] (0,1)--(1,0);
\draw[densely dotted, thin] (2,1)--(3,0);

\draw[->,ultra thick] (5,2)--(6,2) ;

\foreach \i in {3}{
\draw[densely dotted, thin] (\i + 5,3)--(\i + 5,1);
\draw[solid,ultra thick] (\i + 5,1)--(\i + 7,1);
\draw[densely dotted, thin] (\i + 6,0)--(\i + 8,0);
\draw[dashed,ultra thin] (\i + 6,2)--(\i + 8,2);

\draw[solid,ultra thick] (\i + 5,3)--(\i + 6,2);
\draw[dashed,ultra thin](\i + 6,2)--(\i + 6,0);
\draw[densely dotted, thin] (\i + 7,3)--(\i + 7,1);
\draw[solid,ultra thick] (\i + 8,0)--(\i + 8,2);

\draw[densely dotted, thin] (\i + 8,2)--(\i + 7,3);
\draw[dashed,ultra thin] (\i + 7,3)--(\i + 5,3);
\draw[solid,ultra thick] (\i + 5,1)--(\i + 6,0);
\draw[dashed,ultra thin] (\i + 7,1)--(\i + 8,0);
}

\foreach \i in {0,8}{
\draw[solid,ultra thick] (\i + 2,1)--(\i + 3.5,1) .. controls (\i + 4,3)  .. (\i + 2, 4)--(\i + 2,3) -- (\i + 3.5,3) .. controls (\i + 1,3.4) .. (\i + -1, 3) -- (\i + 0,3);
\draw[solid,ultra thick] (\i + 3,2) -- (\i + 3.5,2) .. controls (\i + 1,2.4) .. (\i + -.5, 2) -- (\i + -.5,0) -- (\i + 1,0);
\draw[solid,ultra thick] (\i + 3,0) -- (\i + 3.5,0) .. controls (\i + 5,3) .. (\i + 1, 2.5) -- (\i + 1,2);
}
\end{tikzpicture}

%  \subcaption{X \label{X}}
\end{subfigure}
%%%%%%%%%%%%%%%%%%%%%%%
\newline
\begin{subfigure}{ \textwidth}
  \centering

\begin{tikzpicture}
\foreach \i in {0,2}{
	\foreach \j in {1,3}{
		\foreach \k in {0,1}{
		        \node[vertex, fill, minimum size = 3pt] at (\i + \k, \j - \k) {};
      		        }
	}
}
\foreach \i in {8,10}{
	\foreach \j in {1,3}{
		\foreach \k in {0,1}{
		        \node[vertex, fill, minimum size = 3pt] at (\i + \k, \j - \k) {};
      		        }
	}
}
\node[vertex, fill, minimum size = 6pt] at (0, 3) {};
\node[vertex, fill, minimum size = 6pt] at (8, 3) {};

\node[vertex, fill, minimum size = 6pt] at (0, 3) {};    
\draw[->,ultra thick] (-3,4)--(-2,4) node[right]{$1$};
\draw[->,ultra thick] (-3,4)--(-3,3) node[below]{$2$};
\draw[->,ultra thick] (-3,4)--(-2.5,3.5) node[below]{$3$};

\draw[solid,ultra thin] (-3,2)--(-3,1) node[below]{$X$};
\draw[dashed,ultra thick] (-2.5,2)--(-2.5,1) node[below]{$Y$};
\draw[densely dotted,thin] (-2,2)--(-2,1) node[below]{$Z$};

\draw[solid,ultra thin] (0,3)--(0,1)--(2,1);
\draw[solid,ultra thin] (1,0)--(3,0);
\draw[solid,ultra thin] (1,2)--(3,2);

\draw[dashed,ultra thick] (0,3)--(1,2)--(1,0);
\draw[dashed,ultra thick] (2,3)--(2,1);
\draw[dashed,ultra thick] (3,0)--(3,2);

\draw[densely dotted, thin] (3,2)--(2,3)--(0,3);
\draw[densely dotted, thin] (0,1)--(1,0);
\draw[densely dotted, thin] (2,1)--(3,0);

\draw[->,ultra thick] (5,2)--(6,2) ;

\foreach \i in {3}{
\draw[densely dotted, thin] (\i + 5,3)--(\i + 5,1);
\draw[solid,ultra thin] (\i + 5,1)--(\i + 7,1);
\draw[densely dotted, thin] (\i + 6,0)--(\i + 8,0);
\draw[dashed,ultra thick] (\i + 6,2)--(\i + 8,2);

\draw[solid,ultra thin] (\i + 5,3)--(\i + 6,2);
\draw[dashed,ultra thick](\i + 6,2)--(\i + 6,0);
\draw[densely dotted, thin] (\i + 7,3)--(\i + 7,1);
\draw[solid,ultra thin] (\i + 8,0)--(\i + 8,2);

\draw[densely dotted, thin] (\i + 8,2)--(\i + 7,3);
\draw[dashed,ultra thick] (\i + 7,3)--(\i + 5,3);
\draw[solid,ultra thin] (\i + 5,1)--(\i + 6,0);
\draw[dashed,ultra thick] (\i + 7,1)--(\i + 8,0);
}

\foreach \i in {0,8}{
\draw[dashed,ultra thick] (\i + 1,0) -- (\i + 1,-.5) .. controls (\i + -.2,-1) .. (\i + -.8, 1.8) -- (\i + 0,1) -- (\i + 0,-.5) .. controls (\i + -.5 ,2) .. (\i + 0, 4) -- (\i + 0,3);
\draw[dashed,ultra thick] (\i + 2,1) -- (\i + 2,-1) .. controls (\i + 2.4,2.5) .. (\i +  2,3.5) -- (\i + 3,2.5) -- (\i + 3,2);
\draw[dashed,ultra thick] (\i + 3,0) -- (\i + 3,-1) .. controls (\i + 4,4.5) .. (\i + 1.5,3.5)-- (\i + 2,3);

}
\end{tikzpicture}

  %\subcaption{Y \label{Y}}
\end{subfigure}
%%%%%%%%%%%%%%%%%%%%%%%
%\newline
%\hfill
\begin{subfigure}{ \textwidth}
  \centering

\begin{tikzpicture}
\foreach \i in {0,2}{
	\foreach \j in {1,3}{
		\foreach \k in {0,1}{
		        \node[vertex, fill, minimum size = 3pt] at (\i + \k, \j - \k) {};
      		        }
	}
}
\foreach \i in {8,10}{
	\foreach \j in {1,3}{
		\foreach \k in {0,1}{
		        \node[vertex, fill, minimum size = 3pt] at (\i + \k, \j - \k) {};
      		        }
	}
}
\node[vertex, fill, minimum size = 6pt] at (0, 3) {};
\node[vertex, fill, minimum size = 6pt] at (8, 3) {};

\node[vertex, fill, minimum size = 6pt] at (0, 3) {};    
\draw[->,ultra thick] (-3,4)--(-2,4) node[right]{$1$};
\draw[->,ultra thick] (-3,4)--(-3,3) node[below]{$2$};
\draw[->,ultra thick] (-3,4)--(-2.5,3.5) node[below]{$3$};

\draw[solid,ultra thin] (-3,2)--(-3,1) node[below]{$X$};
\draw[dashed,ultra thin] (-2.5,2)--(-2.5,1) node[below]{$Y$};
\draw[densely dotted,thick] (-2,2)--(-2,1) node[below]{$Z$};

\draw[solid,ultra thin] (0,3)--(0,1)--(2,1);
\draw[solid,ultra thin] (1,0)--(3,0);
\draw[solid,ultra thin] (1,2)--(3,2);

\draw[dashed,ultra thin] (0,3)--(1,2)--(1,0);
\draw[dashed,ultra thin] (2,3)--(2,1);
\draw[dashed,ultra thin] (3,0)--(3,2);

\draw[densely dotted, thick] (3,2)--(2,3)--(0,3);
\draw[densely dotted, thick] (0,1)--(1,0);
\draw[densely dotted, thick] (2,1)--(3,0);

\draw[->,ultra thick] (5,2)--(6,2) ;

\foreach \i in {3}{
\draw[densely dotted, thick] (\i + 5,3)--(\i + 5,1);
\draw[solid,ultra thin] (\i + 5,1)--(\i + 7,1);
\draw[densely dotted, thick] (\i + 6,0)--(\i + 8,0);
\draw[dashed,ultra thin] (\i + 6,2)--(\i + 8,2);

\draw[solid,ultra thin] (\i + 5,3)--(\i + 6,2);
\draw[dashed,ultra thin](\i + 6,2)--(\i + 6,0);
\draw[densely dotted, thick] (\i + 7,3)--(\i + 7,1);
\draw[solid,ultra thin] (\i + 8,0)--(\i + 8,2);

\draw[densely dotted, thick] (\i + 8,2)--(\i + 7,3);
\draw[dashed,ultra thin] (\i + 7,3)--(\i + 5,3);
\draw[solid,ultra thin] (\i + 5,1)--(\i + 6,0);
\draw[dashed,ultra thin] (\i + 7,1)--(\i + 8,0);
}

\foreach \i in {0,8}{
\draw[densely dotted, thick] (\i +  3,2) -- (\i +  4,1).. controls (\i + 1.4,0) .. (\i + .3,2) -- (\i + 1,2)--(\i + 1.4,1.6) .. controls (\i + 1,3) .. (\i + -.5,3.5) -- (\i + 0,3);
\draw[densely dotted, thick] (\i +  -1,1) -- (\i + 0,1);
\draw[densely dotted, thick] (\i +  1,0) -- (\i + 1.8,-.8) .. controls (\i + 0,0) .. (\i + -.4,1.4) -- (\i + 1.6,1.4) -- (\i + 2,1);
\draw[densely dotted, thick] (\i +  3,0) -- (\i + 4,-1) .. controls (\i + -1,-1) .. (\i + -1,1);
}

\end{tikzpicture}

%  \subcaption{Z \label{Z}}
\end{subfigure}

%%%%%%%%%%%%%%%%%%%%%%%
\caption{The cycles merged during the merge operation}\label{p3}
%\draw[solid,ultra thin] (0,3)--(-1,3);
\end{figure}
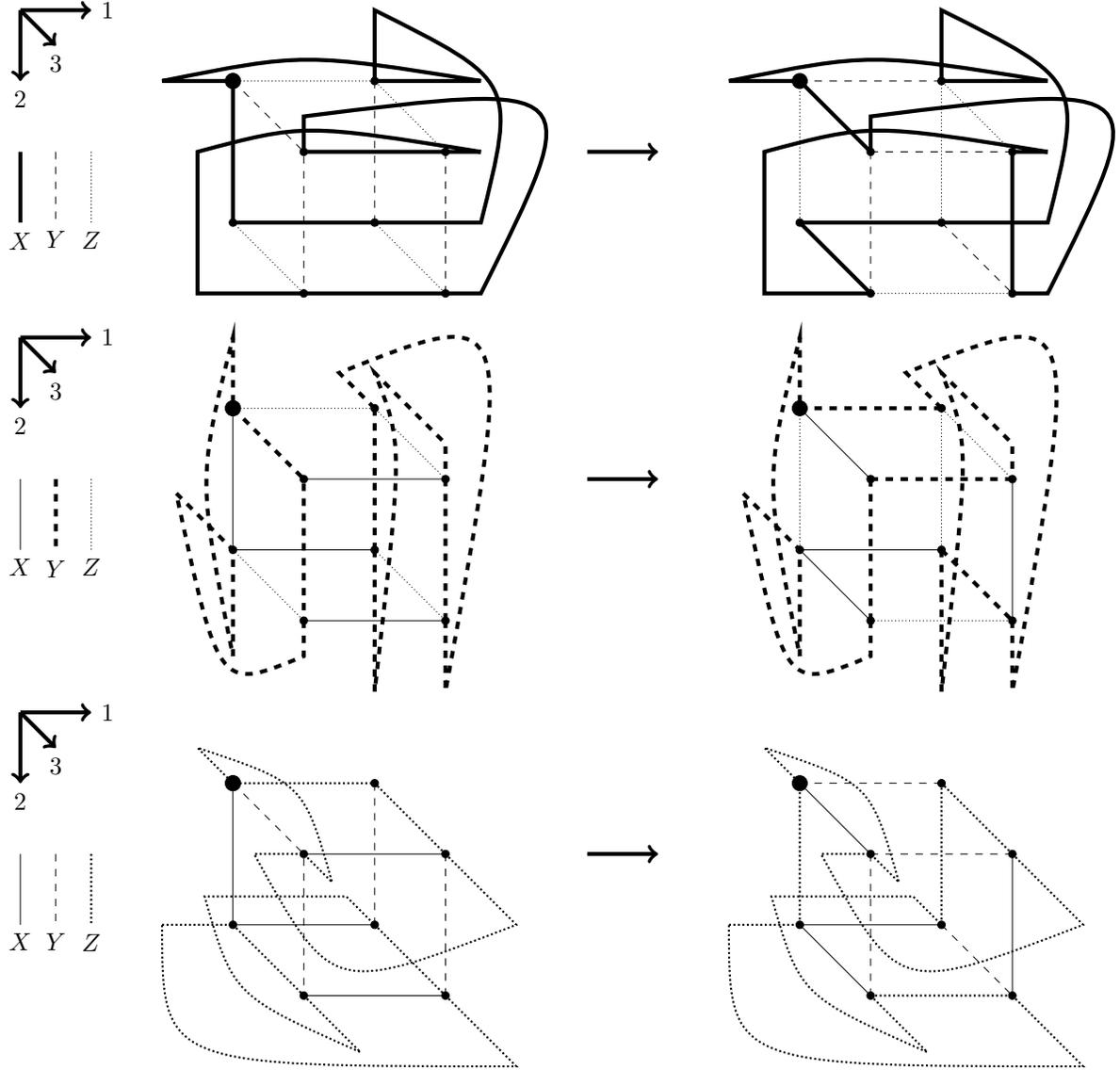

Of course, we need another $4^n-2$ of these merge operations, and as we progress, the structures of the cycles change, which could possibly cause a merge operation to ``fail'' to combine six cycles into three. Hence, Lemma~\ref{l3} is crucial.

\begin{definition}
For $i \le j$ let $[i,j]$ be the set $\{i,i+1, \ldots , j \}$. For $0 \le i < 4^n$, define $Z^n_i$ to be the set of vertices of $G_{n,3}$ that have their 3rd coordinate equal to $i$. Finally, let
\begin{align*}
Z^n_{[i,j]} = \bigcup\limits_{i \le k \le j} Z^n_k
\end{align*}
\end{definition}

\begin{definition}
Let $C$ be a cycle and $S$ be a subset of $V(C)$.
The \emph{$C$-necklace-order with respect to $S$} is the order in which the vertices of $S$ appear in $C$.
As its name suggests, shifting or reversing the direction of $C$ does not change its order (with respect to any vertex set).
\end{definition}

\begin{observation} \label{obs3}
Let $v = (x,y,z)$ be the origin of a type-I cube $L$.
Figure~\ref{p4} shows that the $X \cap Z^n_x$-necklace-order with respect to $Z^n_x$ (before merge) is the same as $X \cap Z^n_{[x,x+1]}$-necklace-order with respect to $Z^n_x$ (after merge).
Indeed, the only change to $X \cap Z^n_x$ is the removal of the edge $uv$, which is replaced by a detour through $Z^n_{x+1}$.
This augmentation does not change the order of vertices of $Z^n_x$.
\end{observation}

\begin{figure}[H]
\centering
\scalebox{.9}{
\begin{tikzpicture}
\foreach \i in {0,2}{
	\foreach \j in {1,3}{
		\foreach \k in {0,1}{
		        \node[vertex, fill, minimum size = 3pt] at (\i + \k, \j - \k) {};
      		        }
	}
}
\foreach \i in {8,10}{
	\foreach \j in {1,3}{
		\foreach \k in {0,1}{
		        \node[vertex, fill, minimum size = 3pt] at (\i + \k, \j - \k) {};
      		        }
	}
}
\node[vertex, fill, minimum size = 6pt] at (0, 3) {};
\node[vertex, fill, minimum size = 6pt] at (8, 3) {};

\foreach \i in {0,8}{
\node at (\i,3.5) {$v$};
\node at (\i + .2,1.3) {$u$};
}

\node[vertex, fill, minimum size = 6pt] at (0, 3) {};    
\draw[->,ultra thick] (-3,4)--(-2,4) node[right]{$1$};
\draw[->,ultra thick] (-3,4)--(-3,3) node[below]{$2$};
\draw[->,ultra thick] (-3,4)--(-2.5,3.5) node[below]{$3$};

\draw[densely dotted, thick] (-3,2)--(-3,1) node[above right]{$X \cap Z^n_x$};
\draw[dashed, thick] (-3,1)--(-3,0) node[above right]{$X \cap Z^n_{x+1}$};
\draw[densely dashdotdotted,thick] (-3,0)--(-3,-1) node[above right]{$X \cap Z^n_{[x,x+1]}$};

\draw[densely dotted, thick] (0,3)--(0,1)--(2,1);
\draw[dashed, thick] (1,0)--(3,0);
\draw[dashed, thick] (1,2)--(3,2);

\draw[solid,ultra thin] (0,3)--(1,2)--(1,0);
\draw[solid,ultra thin] (2,3)--(2,1);
\draw[solid,ultra thin] (3,0)--(3,2);

\draw[solid, ultra thin] (3,2)--(2,3)--(0,3);
\draw[solid, ultra thin] (0,1)--(1,0);
\draw[solid, ultra thin] (2,1)--(3,0);

\draw[->,ultra thick] (5.3,2)--(6.3,2) ;

\foreach \i in {3}{
\draw[solid, ultra thin] (\i + 5,3)--(\i + 5,1);
\draw[densely dashdotdotted,thick] (\i + 5,1)--(\i + 7,1);
\draw[solid, ultra thin] (\i + 6,0)--(\i + 8,0);
\draw[solid,ultra thin] (\i + 6,2)--(\i + 8,2);

\draw[densely dashdotdotted,thick] (\i + 5,3)--(\i + 6,2);
\draw[solid,ultra thin](\i + 6,2)--(\i + 6,0);
\draw[solid, ultra thin] (\i + 7,3)--(\i + 7,1);
\draw[densely dashdotdotted,thick] (\i + 8,0)--(\i + 8,2);

\draw[solid, ultra thin] (\i + 8,2)--(\i + 7,3);
\draw[solid,ultra thin] (\i + 7,3)--(\i + 5,3);
\draw[densely dashdotdotted,thick] (\i + 5,1)--(\i + 6,0);
\draw[solid,ultra thin] (\i + 7,1)--(\i + 8,0);
}

\foreach \i in {0}{
\draw[densely dotted, thick] (\i + 2,1)--(\i + 3.5,1) .. controls (\i + 4,3)  .. (\i + 2, 4)--(\i + 2,3) -- (\i + 3.5,3) .. controls (\i + 1,3.4) .. (\i + -1, 3) -- (\i + 0,3);
\draw[dashed, thick] (\i + 3,2) -- (\i + 3.5,2) .. controls (\i + 1,2.4) .. (\i + -.5, 2) -- (\i + -.5,0) -- (\i + 1,0);
\draw[dashed, thick] (\i + 3,0) -- (\i + 3.5,0) .. controls (\i + 5,3) .. (\i + 1, 2.5) -- (\i + 1,2);
}

\foreach \i in {8}{
\draw[densely dashdotdotted,thick] (\i + 2,1)--(\i + 3.5,1) .. controls (\i + 4,3)  .. (\i + 2, 4)--(\i + 2,3) -- (\i + 3.5,3) .. controls (\i + 1,3.4) .. (\i + -1, 3) -- (\i + 0,3);
\draw[densely dashdotdotted,thick] (\i + 3,2) -- (\i + 3.5,2) .. controls (\i + 1,2.4) .. (\i + -.5, 2) -- (\i + -.5,0) -- (\i + 1,0);
\draw[densely dashdotdotted,thick] (\i + 3,0) -- (\i + 3.5,0) .. controls (\i + 5,3) .. (\i + 1, 2.5) -- (\i + 1,2);
}

\end{tikzpicture}
}
\caption{The $X \cap Z^n_x$-necklace-order with respect to $Z^n_x$ (left) is the same as the $X \cap Z^n_{[x,x+1]}$-necklace-order with respect to $Z^n_x$ (right).
}\label{p4}
\end{figure}
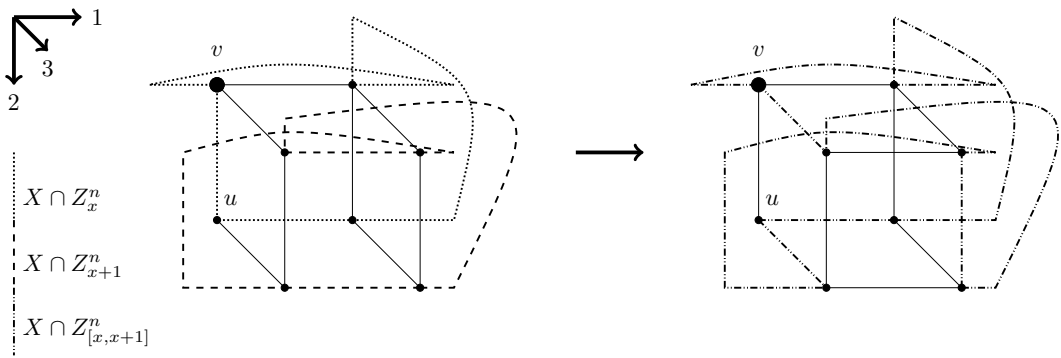

\begin{lemma} \label{l3}
Suppose that we just decomposed the edge set of $G_{n,3}$ with the method given in Lemma~\ref{l2}.
Let $v = (x,y,z)$ and $v' = (x',y',z')$ be such that $x+y+z = x'+y'+z' = -1 \pmod{4^n}$ and $z = z'+1 \pmod{4^n}$, and let $L$ and $L'$ be the type-I cubes with origins at $v$ and $v'$, respectively.
If we merge $L$ first and then $L'$, we reduce $c_X$ by 2.
\end{lemma}
\begin{proof}
We saw in Observation~\ref{obs2} that a single merge operation always succeeds.
Suppose that we have merged $L$, so that $Z^n_{x+1}$ and $Z^n_{x+2}$ have merged into $Z^n_{[x+1,x+2]}$, and we are about to merge $L'$.
By Observation~\ref{obs3}, the order of vertices in $Z^n_{x+1}$ has not changed, so merging $L'$ will successfully combine $Z^n_{[x+1,x+2]}$ and $Z^n_x$ into a single cycle $Z^n_{[x,x+2]}$.
\end{proof}

We now specify a condition under which all the merge operations are guaranteed to succeed.

\begin{definition}
Let $S \subseteq [0,4^n-1]^3$.
We say that $S$ is a \emph{merging set} if it satisfies the following:
\begin{itemize}
\item $|S| = 4^n-1$,
\item Members $(x,y,z)$ of $S$ satisfy $x+y+z = -1 \pmod{4^n}$, and
\item Distinct members $(x,y,z)$ and $(x',y',z')$ of $S$ satisfy $x \ne x'$, $y \ne y'$, and $z \ne z'$.
\end{itemize}
\end{definition}

We need $4^n-1$ merge operations, each merging six cycles into three.
In order for all these operations to successfully take place, it suffices for the type-I cubes to have their origins in a merging set.
This we show next.
%Let $G_X$ be the graph having the vertex set of $G_{n,3}$ and having $X$ as its edge set. Define $G_Y$ and $G_Z$ similarly.

\begin{lemma}
\label{l4}
Consider the following procedure:
\begin{enumerate}[label=\roman*]
\item Decompose the edge set of $G_{n,3}$ with the method given in Lemma~\ref{l2}. \label{s1}
\item Select a merging set $S$. \label{s2}
\item Recognize the $4^{n}-1$ type-I cubes that have their origins in $S$. \label{s3}
\item Replace each type-I cube with a type-II cube. \label{s4}
\end{enumerate}
The following statements hold:
\begin{enumerate}
\item After completing step~\ref{s1}, we have $c_X = c_Y = c_Z = 4^n$, with different components of $X$ being $Z^n_i$'s, different components of $Y$ being $X^n_i$'s, and different components of $Z$ being $Y^n_i$'s.
\item The type-I cubes are pairwise disjoint.
\item After fixing $S$ in step~\ref{s2} and the type-I cubes in step~\ref{s3}, throughout step~\ref{s4} \label{ps3}
\begin{itemize}
\item For a fixed $i$, the vertices of $Z^n_i$ remain in the same component of $X$, the vertices of $X^n_i$ remain in the same component of $Y$, and the vertices of $Y^n_i$ remain in the same component of $Z$, and
\item Every merge operation reduces $c_X$, $c_Y$, and $c_Z$ by 1.
\end{itemize}
In particular, after finishing step~\ref{s4}, we have an H.D. for $G_{n,3}$.
\end{enumerate}
\end{lemma}

\begin{proof}
\noindent
\begin{enumerate}
\item This was shown in Lemma~\ref{l2}.
\item Suppose there exist $(x_1,x_2,x_3)$ and $(x'_1,x'_2,x'_3)$ in $S$ such that their corresponding type-I cubes have some vertex in common, so that for some $(i_1,i_2,i_3)$ and $(i'_1,i'_2,i'_3)$ in $\{0,1\}^3$ we have
\begin{align*} 
(x_1,x_2,x_3)+(i_1,i_2,i_3) =(x'_1,x'_2,x'_3)+(i'_1,i'_2,i'_3) \  \pmod{4^n}.
\end{align*} 
It follows that $x_r+i_r=x'_r+i'_r \pmod{4^n}$ for each $1 \le r \le 3$.
Adding these congruences we get $x_1 + x_2 + x_3+ i_1 + i_2 + i_3 =x'_1 + x'_2 + x'_3 + i'_1 + i'_2 + i'_3 \  \pmod{4^n}$, but $x_1+x_2+x_3=x'_1+x'_2+x'_3=-1 \pmod{4^n}$, so we obtain $i_1+i_2+i_3=i'_1+i'_2+i'_3 \pmod{4^n}$, and in particular, $i_1+i_2+i_3=i'_1+i'_2+i'_3 \pmod{2}$.
This implies that $i_r=i'_r$ for some $r$, meaning that $x_r=x'_r \pmod{4^n}$ for the same $r$.
This gives $x_r=x'_r$, which contradicts the assumption that $S$ is a merging set.
\item Due to the symmetry involved in (\ref{ps3}), it suffices to prove the assertions in just one direction, that is, to prove
\begin{itemize}
\item The vertices of $Z^n_i$ remain in the same component of $X$, and
\item Every merge operation reduces $c_X$ by 1.
\end{itemize}
Without loss of generality, suppose that $S = \{ v_0, v_1, \ldots, v_{4^n-2} \} $, where $v_i = (x_i,y_i, i)$, and let $L_i$ be the type-I cube with origin at $v_i$.
Because of (2), the order in which we merge the cubes does not matter, so for the sake of simplicity, assume that $L_{4^n-2}$ is merged first, $L_{4^n-3}$ is merged next, and so on.

We proceed by induction.
The base case is satisfied due to Obsevation~\ref{obs2} and Lemma~\ref{l3}.
Suppose that we have merged cubes $L_{4^n-2}$ to $L_i$.
The induction hypothesis states that we have cycles $X \cap Z^n_{1}, X \cap Z^n_{2}, \ldots, X \cap Z^n_{i-1}$, and a long cycle $X \cap Z^n_{[i,4^n-1]}$.
It also states that the vertices of $Z^n_i$ have been in the same component of $X$ together throughout step \ref{s4}.
By Observation~\ref{obs3}, the order of the vertices in $Z^n_i$ has not changed yet, so merging $L_{i-1}$ will combine $X \cap Z^n_{[i,4^n-1]}$ and $X \cap Z^n_{i-1}$ into a single cycle $X \cap Z^n_{[i-1,4^n-1]}$.
It is clear that the vertices of $Z^n_{i}$ have remained and will remain in the same component of $X$.
\end{enumerate}

\end{proof}

%%%%%%%%%%%%%%%%%%%%%%%%%%%%%%%%%%%%%%%%%%%%%%%%%%%%%%

\subsection{Deriving an H.D. for $Q_{6n}$ from an H.D. for $Q_{2n}$}\hspace*{\fill}

Combining the Hamilton decompositions for $Q_{2n}$ and $G_{n,3}$ is very similar to the 2-dimensional case.
We think of $G_{n,3}$ as a 3-dimensional cyclic grid, and assign coordinates like $(x,y,z)$ to its vertices.
\begin{definition}
Let $E$ be a directed Hamilton cycle in $Q_{2n}$.
A \emph{3-dimensional seating} of $Q_{6n}$ onto $G_{n,3}$ via $E$, is a representation of the vertices of $Q_{6n}$ by assigning them integral coordinates as follows:
\begin{enumerate}
    \item Consider $E$ and its positive direction.
    Assign $0$ to $\bm{0}$, assign $1$ to the next vertex in $E$, and continue until $4^n-1$ is assigned to the last vertex of $E$.
    \item Induce the order of $E$ onto $Q_{2n}$, so that each vertex has the same order in either graph. \label{d32}
    \item Using the coordinates in (\ref{d32}), assign coordinates to every member of $Q_{6n} = Q_{2n} \Box Q_{2n} \Box Q_{2n}$.
    Put the vertices on the 3-dimensional grid using their coordinates.
\end{enumerate}
Using the natural order of $E$, we have mapped the vertices of $Q_{6n}$ onto $G_{n,3}$ and recognized $Q_{6n}$ as a supergraph of $G_{n,3}$.
%Any subgraph of $G_{n,3}$, therefore, is also a subgraph of $Q_{6n}$.
If $H$ is a directed Hamilton cycle in $G_{n,3}$, the \emph{3-dimensional directed Hamilton cycle derived from $E$ and $H$}, denoted by $g(E,H)$, is a Hamilton cycle in $Q_{6n}$ and is defined in the natural way:
\begin{enumerate}
\item 3-dimensionally seat $Q_{6n}$ onto $G_{n,3}$ via $E$.
\item $Q_{6n}$ has $3n$ axes $1,2, \ldots , 3n$, while $G_{n,3}$ has an $x$-axis, a $y$-axis, and a $z$-axis.
The axes $1,2, \ldots , n$ are in direction $x$, the axes $n+1,n+2, \ldots , 2n$ are in direction $y$, and the axes $2n+1,2n+2, \ldots , 3n$ are in direction $z$.
\item $g(E,H)$ has the same edges in the supergraph $Q_{2n} \Box Q_{2n} \Box Q_{2n}$ as $H$ has in the subgraph $G_{n,3}$.
\end{enumerate}

\end{definition}

\begin{lemma}
\label{l5}
Let $H_1$ and $H_2$ be two disjoint Hamilton cycles in  $G_{n,3}$ and $E_1$ and $E_2$ be two disjoint Hamilton cycles in $Q_{2n}$. Then the four Hamilton cycles $G_1 = g(E_1 , H_1)$, $G_2 = g(E_1,H_2)$, $G_3 = g(E_2 , H_1)$, and $G_4 = g(E_2 , H_2)$ in $Q_{6n}$ are pairwise disjoint.
\end{lemma}
\begin{proof}
We show that $G_1 = g(E_1, H_1)$ is disjoint from the other three cycles $G_2$, $G_3$, and $G_4$.
To achieve this, we 3-dimensionally seat $Q_{6n}$ onto $G_{n,3}$ via $E_1$.
Similar to the $2$-dimensional case, $G_1$ and $G_2$ have all their edges on the grid, while $G_3$ and $G_4$ have all their edges off the grid.
Thus $G_1$ is disjoint from $G_3$ and $G_4$.
Furthermore, $G_1$ and $G_2$ represent $H_1$ and $H_2$, respectively, and $H_1$ and $H_2$ are disjoint, so $G_1$ and $G_2$ must be disjoint as well.
\end{proof}

\begin{corollary} \label{c2}

If $\{ H_1, H_2, H_3 \} $ is an H.D. for $G_{n,3}$ and $\{ E_1 , E_2 , \ldots , E_{n}  \} $ is an H.D. for $Q_{2n}$, then the family $\{ g(E_i , H_j) \mid 1 \le i \le n, 1 \le j \le 3 \}$ is an H.D. for $Q_{6n}$.
The new Hamilton cycles are named $F_1$, $F_2$, \ldots, $F_{3n}$ via $F_{j} = f(E_j,H_1)$, $F_{j+n} = f(E_j,H_2)$, and $F_{j+2n} = f(E_j,H_3)$ for $1 \le j \le n$.
\end{corollary}

%%%%%%%%%%%%%%%%%%%%%%%%%%%%%%%%%%%%%%%%%%%%%%%%%%%%%%%%%%%%%

\subsection{An algorithm for computing an H.D. for $G_{n,3}$}\hspace*{\fill}

In section~\ref{sec41} we saw how to derive a Hamilton decomposition $\{X, Y, Z\}$ from the initial partitioning given by Lemma~\ref{l2}.
We now give an algorithm to compute $X$.
Algorithms for $Y$ and $Z$ are similar.

The idea is to apply the edge decomposition given in Lemma~\ref{l2}, and then proceed from the origin, initially moving in the positive direction of $X$, until we reach a chosen type-I cube (one whose origin belongs to the merging set).
We then recognize the special vertex, take the necessary actions mandated by the merge operation, and continue to walk in $X$.
Figure~\ref{p5} shows all the special vertices and the reasoning behind our actions.
For example, if we reach $m'$ and the current direction is negative, it means that we came from outside of the cube (and not from $m$), so we should go to $m$ and change the direction to positive, so that we move outside in the next step.
If we reach $m'$ and the current direction is positive, however, it means that we came from $m$ (and not from outside), so we should move outside and leave the direction unchanged.

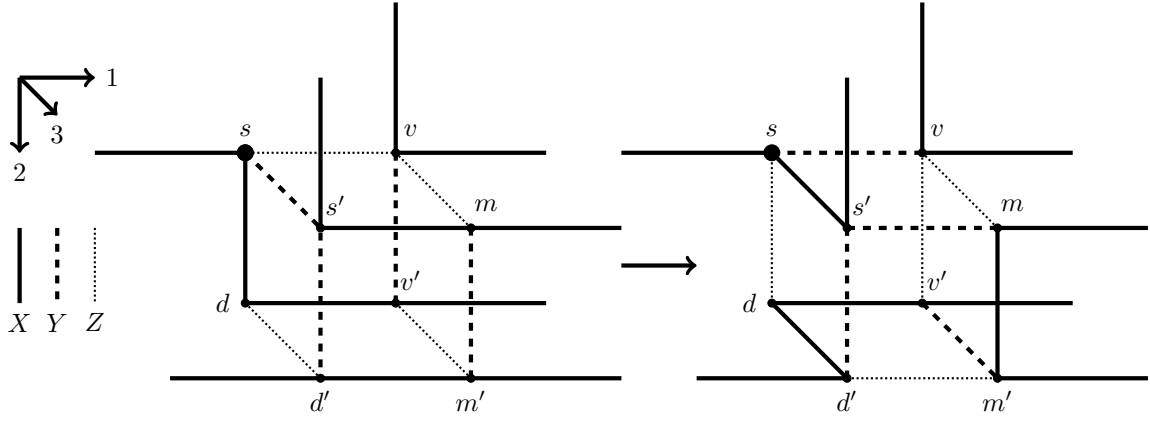
\begin{figure}[H]
\centering

\begin{tikzpicture}
\foreach \i in {0,2}{
	\foreach \j in {1,3}{
		\foreach \k in {0,1}{
		        \node[vertex, fill, minimum size = 3pt] at (\i + \k, \j - \k) {};
      		        }
	}
}
\foreach \i in {7,9}{
	\foreach \j in {1,3}{
		\foreach \k in {0,1}{
		        \node[vertex, fill, minimum size = 3pt] at (\i + \k, \j - \k) {};
      		        }
	}
}
\node[vertex, fill, minimum size = 6pt] at (0, 3) {};
\node[vertex, fill, minimum size = 6pt] at (7, 3) {};

\node[vertex, fill, minimum size = 6pt] at (0, 3) {};    
\draw[->,ultra thick] (-3,4)--(-2,4) node[right]{$1$};
\draw[->,ultra thick] (-3,4)--(-3,3) node[below]{$2$};
\draw[->,ultra thick] (-3,4)--(-2.5,3.5) node[below]{$3$};

\draw[solid,ultra thick] (-3,2)--(-3,1) node[below]{$X$};
\draw[dashed,ultra thick] (-2.5,2)--(-2.5,1) node[below]{$Y$};
\draw[densely dotted,thick] (-2,2)--(-2,1) node[below]{$Z$};

\draw[solid,ultra thick] (0,3)--(0,1)--(2,1);
\draw[solid,ultra thick] (1,0)--(3,0);
\draw[solid,ultra thick] (1,2)--(3,2);

\draw[dashed,ultra thick] (0,3)--(1,2)--(1,0);
\draw[dashed,ultra thick] (2,3)--(2,1);
\draw[dashed,ultra thick] (3,0)--(3,2);

\draw[densely dotted, thick] (3,2)--(2,3)--(0,3);
\draw[densely dotted, thick] (0,1)--(1,0);
\draw[densely dotted, thick] (2,1)--(3,0);

\foreach \i in {0,7}{
\draw[solid,ultra thick] ( \i +0,3)--( \i +-2,3);
\draw[solid,ultra thick] ( \i +1,0)--( \i +-1,0);
\draw[solid,ultra thick] ( \i +1,2)--( \i +1,4);
\draw[solid,ultra thick] ( \i +2,3)--( \i +2,5);
\draw[solid,ultra thick] ( \i +2,3)--( \i +4,3);
\draw[solid,ultra thick] ( \i +3,2)--( \i +5,2);
\draw[solid,ultra thick] ( \i +2,1)--( \i +4,1);
\draw[solid,ultra thick] ( \i +3,0)--( \i +5,0);
}

\draw[->,ultra thick] (5,1.5)--(6,1.5) ;

\foreach \i in {2}{
\draw[densely dotted, thick] ( \i +5,3)--( \i +5,1);
\draw[solid,ultra thick] ( \i +5,1)--( \i +7,1);
\draw[densely dotted, thick] ( \i +6,0)--( \i +8,0);
\draw[dashed,ultra thick] ( \i +6,2)--( \i +8,2);

\draw[solid,ultra thick] ( \i +5,3)--( \i +6,2);
\draw[dashed,ultra thick]( \i +6,2)--( \i +6,0);
\draw[densely dotted, thick] ( \i +7,3)--( \i +7,1);
\draw[solid,ultra thick] ( \i +8,0)--( \i +8,2);

\draw[densely dotted, thick] ( \i +8,2)--( \i +7,3);
\draw[dashed,ultra thick] ( \i +7,3)--( \i +5,3);
\draw[solid,ultra thick] ( \i +5,1)--( \i +6,0);
\draw[dashed,ultra thick] ( \i +7,1)--( \i +8,0);
}

\foreach \i in {0,7}{
\node at ( \i +0, 3.3) {$s$};
\node at ( \i +-.3, 1) {$d$};
\node at ( \i +1.2, 2.3) {$s'$};
\node at ( \i +1, -.3) {$d'$};
\node at ( \i +3.2, 2.3) {$m$};
\node at ( \i +3, -.3) {$m'$};

\node at ( \i +2.2, 3.3) {$v$};
\node at ( \i +2.2, 1.3) {$v'$};
}

\end{tikzpicture}

\caption{A merge operation together with the attached $X$-edges.
The above vertex labelling conforms to that of Algorithm~\ref{a2}.}\label{p5}
\end{figure}

We choose the merging set to be
\begin{align*}
S & = \Bigg\{ \left( 0,0,4^t-1 \right), \left( 1,1,4^t-3 \right), \ldots ,  \left( \frac{4^t}{2}-1, \frac{4^t}{2}-1 ,1 \right), \\ 
& \left( \frac{4^t}{2}, \frac{4^t}{2}+1 , 4^t-2 \right), \left( \frac{4^t}{2}+1,\frac{4^t}{2}+2, 4^t-4 \right), \ldots, \left( 4^t-2, 4^t-1,2 \right) \Bigg\}.
\end{align*}
We choose $S$ like this for two reasons:
\begin{itemize}
\item The origin does not belong to any of the type-I cubes, so we do not need an initial case check.
\item $S$ has all the $x$-coordinates from $0$ to $4^n-2$, so it is easy to check if a coordinate belongs to it.
\end{itemize}
We define five helping sets $S'$, $D$, $D'$, $M$, and $M'$ so that we have instant access to all the special vertices.
Algorithm~\ref{a2} given in Appendix~\ref{ap2} calculates $X$.

\subsection{3-Dimensional Algorithm}\hspace*{\fill}

Just like in the 2-dimensional case, we use the definition of $g(E,H)$ to devise an algorithm for computing an H.D. for $Q_{6n}$.
Algorithm~\ref{a3} is very similar to Algorithm~\ref{a1}, and is given in Appendix~\ref{ap2}.

%\newpage

\section{Highly Symmetric Hamilton Decompositions}
The theory we have developed in the previous chapters can be improved to give us highly symmetric Hamilton decompositions.
Let $\sigma : [1,k] \to [1,k]$ be a permutation.
Then $\sigma$ induces a homomorphism of $G_{n,k}$ by relabelling the axes: The axis previously referred to as $i$ is now called $\sigma(i)$.
More specifically, the vertex $v = (x_1, x_2, \ldots, v_k)$ is mapped to $\sigma(v) = (x_{\sigma^{-1}(1)}, x_{\sigma^{-1}(2)}, \ldots, x_{\sigma^{-1}(k)})$.
As $\sigma$ is a homomorphism, it maps Hamilton cycles to Hamilton cycles.
If $H = e_1 e_2 \ldots e_{4^{nk}}$ is a directed Hamilton cycle in $G_{n,k}$, then $\sigma(H)$ is the Hamilton cycle
\begin{align*}
\sigma(e_1) \sigma(e_2) \ldots \sigma(e_{4^{nk}})
\end{align*}
Note that $\sigma$ maps backward edges to backward edges: If $\sigma(i) = j$, then $\sigma(\overline{i}) = \overline{j}$.
It is worth remembering that $\overline{i}$ stands for an edge from $(x_1, x_2, \ldots, x_{i-1}, x_i, x_{i+1}, \ldots, x_n)$ to $(x_1, x_2, \ldots, x_{i-1}, x_i-1 \pmod{4^n}, x_{i+1}, \ldots, x_n)$.
\begin{definition}
A family $\mathscr{S} = \{ \sigma_1 , \sigma_2 , \ldots , \sigma_k \}$ of $k$ permutations on $[1,k]$ is called a \emph{Latin family} if the matrix $m_{ij} = \sigma_i(j)$ is a Latin square.
We do not differentiate between $\sigma_i$ and the {\it$i$th} row of the matrix.
For the sake of simplicity, we require that $\sigma_1$, the first row of the matrix, is the identity.

Let $T = \{ H_1 , H_2 , \ldots , H_k \} $ be an H.D. for $G_{n,k}$. We say that $T$ is a \emph{Latin Hamilton decomposition} if there exists a Hamilton cycle $H$ in $G_{n,k}$ and a Latin family $\mathscr{S} = \{ \sigma_1 , \sigma_2 \ldots , \sigma_k \}$ of permutations on $[1,k]$ such that
\begin{align*}
H_i = \sigma_i(H) \qquad 1 \le i \le k
\end{align*}
The Hamilton cycle $H$ ($= H_1$)  is then called a \emph{source cycle} for $G_{n,k}$ and the matrix $m_{ij} = \sigma_i(j)$ is called a \emph{source matrix} for $G_{n,k}$.
The pair $(H,M)$ is called a \emph{source pair} for $G_{n,k}$.
\end{definition}

The H.D. given for $G_{n,2}$ in~\ref{sec31} is Latin, but the one given for $G_{n,3}$ in~\ref{sec41} is not necessarily so.
If it is not Latin, we can turn it into one with a small adjustment.

\begin{theorem} \label{t1}
The set $S$ mentioned in Lemma~\ref{l4} step~\ref{s2} can be chosen in such a way that the resulting H.D. is Latin. More specifically, if
\allowdisplaybreaks
\begin{align*}
& S^* = \Bigg\{ \left( 0, \frac{4^t}{2} -1 , \frac{4^t}{2} \right) ,  \left(1, \frac{4^t}{2} -3 , \frac{4^t}{2} +1\right) , \ldots , \left( \frac{4^t -4}{6} , \frac{4^t +2}{6}, \frac{4 \times 4^t -4}{6} \right) , \\
& \left( \frac{4 \times 4^t +2}{6} , \frac{5 \times 4^t +4}{6} , \frac{4^t}{2} -2 \right) ,  \left( \frac{4 \times 4^t +2}{6} +1 , \frac{5 \times 4^t +4}{6} +1 , \frac{4^t}{2} -4 \right) , \ldots , \\
& \left( \frac{5 \times 4^t -2}{6} -1 , 4^t -1 , \frac{4^t +2}{6} +1\right) \Bigg\}
\end{align*}
and
\begin{align}\label{f3}
S = \left\{ \left(x,y,z\right) \bigm\vert \left(x,y,z\right) \in S^*, \text{ or } \left(y,z,x\right)  \in S^*, \text{ or } \left(z,x,y\right) \in S^* \right\},
\end{align}
then the resulting H.D. is Latin.
\end{theorem}
\begin{proof}
We show that it suffices for $S$ to have the following property:
\begin{center}
If $(x,y,z) \in S$, then $(y,z,x) \in S$ and $(z,x,y) \in S$.
\end{center}
To see this, consider $G_{n,3}$ after completion of Lemma~\ref{l4} step~\ref{s1}.
%Fix $E = E(G_X)$.
Let $\sigma_i: [1,3] \to [1,3]$ be defined via $\sigma_i(j) = i + j -1 \pmod{3}$ for $i$ and $j$ in $[1,3]$.
%If $P$ is a path from the origin to the vertex $(x,y,z)$, then $\sigma_2(P)$ is a path from the origin to the vertex $(z,x,y)$ and $\sigma_3(P)$ is a path from the origin to the vertex $(y,z,x)$.
It is not hard to see that
\begin{align}\label{f4}
\sigma_2(X) = Y \text{ and } \sigma_3(X) = Z.
\end{align}
%If we permute the axes of $G_{n,3}$ and place the third axis in place of second axis, second axis in place of first axis, and first axis in place of third axis, then $E$ completely maps onto $E(G_Z)$. Call this process a \emph{negative shift on axes}. Note that an edge which was previously denoted by $i$ is now denoted by $i-1$  or $i+1$. If we do this shift again, $E$ will map onto $E(G_Y)$. Doing this again will map $E$ back onto $X$. \\
We wish to show that the relations given in~\ref{f4} remain valid after completion of Lemma~\ref{l4} step~\ref{s4}.
To achieve this, we merge the cubes three at a time and use induction.

Suppose that $u_1 = (x,y,z)$, $u_2 = \sigma_2(u_1) = (z,x,y)$, and $u_3 = \sigma_3(u_1) = (y,z,x)$ belong to $S$, and let $L_1$, $L_2$, and $L_3$ be type-I cubes with their origins at $u_1$, $u_2$, and $u_3$, respectively.
By the induction hypothesis, we know that~\ref{f4} is valid before merging $L_1$, $L_2$, and $L_3$.

Since $u_2 = \sigma_2(u_1)$, we have $L_2 = \sigma_2(L_1)$, and because $u_3 = \sigma_3(u_1)$, we get $L_3 = \sigma_3(L_1)$.
Furthermore, analyzing the merge operator gives $\sigma_2 \left( L_1 \cap X \right) = L_2 \cap Y$ and $\sigma_3 \left( L_1 \cap X \right) = L_3 \cap Z$.
%If we do a negative shift on axes of $G_{n,3}$, edges of $L_1$ will completely  map onto edges of $L_2$.
%Furthermore, by analyzing the merge operator, we can easily note that the sets $E(L_1) \cap H_X$, $E(L_1) \cap H_Z$, and $E(L_1) \cap H_Y$ will completely map onto $E(L_2) \cap H_Z$, $E(L_2) \cap H_Y$, and $E(L_2) \cap H_X$, respectively.
%Therefore, if $S$ is such that for every vertex $(x,y,z)$ in $S$, the vertex $(z,x,y)$ is also in $S$, then we have $\sigma_2(X) = Y$.
%Therefore having one Hamilton cycle and negative shifting the axes gives us the other two cycles.
This means that~\ref{f4} is valid after merging the three cubes.
Therefore $X$ (after finishing Lemma~\ref{l4} step~\ref{s4}) is a source cycle for $G_{n,3}$ in the H.D. $\{X,Y,Z \}$, and its source matrix is
$ \left[ \begin{array}{c|c|c} 1 & 2 & 3 \\ \hline 2 & 3 & 1 \\ \hline 3 & 1 & 2 \end{array} \right]$.
\end{proof}

We may modify Algorithm~\ref{a1} to take source pairs for $Q_{2n}$ and $G_{n,2}$ and produce a source pair for $Q_{4n}$.
We may also modify Algorithm~\ref{a3} to take source pairs for $Q_{2n}$ and $G_{n,3}$ and produce a source pair for $Q_{6n}$.
Algorithms~\ref{a4} and~\ref{a5} are the Latin counterparts to Algorithms~\ref{a1} and~\ref{a3}, respectively, and are given in Appendix~\ref{ap2}.
We may also specify that Algorithm~\ref{a2} takes a suitable merging set (\ref{f3}) so that it produces a source cycle for $G_{n,3}$.
Hence, it is not necessary to give a Latin counterpart to Algorithm~\ref{a2}.
%It does not produce highly sy The second one, given in \cite{MR1989980}, uses the H.D. given in (\ref{f2}) to do the job.

\begin{theorem} \label{t2}
If $ \{ H_1 , H_2  \}$ and $\{ E_1 , E_2 , \ldots , E_n \}$ mentioned in Corollary~\ref{c1} are Latin, then the resulting Hamilton decomposition $ \{ f(E_i , H_j) \mid 1 \le i \le n \text{ and } 1 \le j \le 2 \} $ is also Latin.
\end{theorem}
\begin{proof}
Let $E_1$, our source cycle for $Q_{2n}$, have source matrix $M$. For $G_{n,2}$, the cycle $H_1$ is a source cycle and has source matrix 
$
\left[
\begin{array}{c|c}
1 & 2 \\ \hline
2 & 1
\end{array}\right]
$.
We show that $F_1 = f(E_1,H_1)$ is a source cycle for $Q_{4n}$ with
$M' =
\left[
\begin{array}{c|c}
M & M+n \\ \hline
M+n & M
\end{array}\right]
$
as its source matrix, where $M+n$ is obtained from $M$ by adding $n$ to every entry.

Our proof is based on Algorithm~\ref{a1}.
In Appendix \ref{ap3} it is shown that Algorithm~\ref{a1} computes $f(E,H)$ correctly.
We know that, for $1 \le j \le n$, this algorithm stores $f(E_j,H_1)$ and $f(E_j,H_2)$ as $F_j$ and $F_{j+n}$, respectively.
The dimension of the {\it$i$th} edge of $F_j$ is stored in $f[j-1][i-1][0]$ and its direction is stored in $f[j-1][i-1][1]$.
Due to line \ref{a1l9} in the algorithm and the fact that $H_1$ and $H_2$ make a Latin decomposition, for every $1 \le i \le 4^{2n}$, either all the $F_i$'s have a forward edge in the {\it$i$th} position or all the $F_i$'s have a backward edge in the {\it$i$th} position.
So the directions of the edges are as required and we only need to focus on their dimensions.

To show that the edge dimensions are as we want, we define a $2n$ by $2n$ matrix $Q$ via
\begin{align*}
q_{i,j} = t \text{ if there is some } 0 \le s < 4^{2n} \text{ such that } f[0][s][0] = j \text{ and } f[i+1][s][0] = t.
\end{align*}
We show that
\begin{itemize}
    \item $Q$ is well-defined, and
    \item $Q = M'$.
\end{itemize}
This would complete the proof of the theorem.

For $1 \le i \le 2n$, let $S_i$ be the set of edge numbers in $F_1$ with dimension $i$.
More precisely
\begin{align*}
    S_i = \{ j \vert 0 \le j < 4^{2n} \text{ and } f[0][j][0] = i \}
\end{align*}
Suppose that $1 \le v \le n$ and let $s \in S_v$.
Lines~\ref{a1l10},~\ref{a1l12}, and~\ref{a1l23} say that, for $j=0$, $i=s$, and $k=0$, we have dim~$=0$ and that for $u = c[k][\text{dim}]$ we have $f[0][s][0] = e[0][u][0]$, but $s \in S_v$, so we have $f[0][s][0] = e[0][u][0] = v = m'_{1,v} = m_{1,v}$.
Therefore, $q_{1,v}$ is well-defined and is equal to $m_{1,v}$.
Again, due to lines~\ref{a1l10},~\ref{a1l12}, and~\ref{a1l23}, for $0 \le w < n$, putting $j=w$ but keeping the same $i$ and $k$, we have the same $u$, and thus $f[w][s][0] = e[w][u][0] = m_{w+1,v}$.
This means that $q_{w+1,v}$ is well-defined as is equal to $m_{w+1,v}$.
Since $w$ and $v$ were arbitrary in $[0,n-1]$ and $[1,n]$, respectively, we get $q_{w+1,v} = m_{w+1,v}$ for $1 \le v \le n$ and $0 \le w < n$.

A similar argument for the other cases shows that
\begin{itemize}
    \item for $n+1 \le v \le 2n$ and $1 \le w \le n$ we have $q_{w,v} = m_{w,v-n} + n$,
    \item for $1 \le v \le n$ and $n+1 \le w \le 2n$ we have $q_{w,v} = m_{w-n,v} + n$, and
    \item for $n+1 \le v \le 2n$ and $n+1 \le w \le 2n$ we have $q_{w,v} = m_{w-n,v-n}$.
\end{itemize}
This shows that $Q$ is well-defined and $Q = M'$.

\end{proof}
As a corollary, we have the following important result.
\begin{corollary} \label{c3}
If $Q_{2n}$ has a source cycle, so does $Q_{4n}$.
\end{corollary}
\begin{theorem} \label{t3}
If $\{ H_1 , H_2  , H_3 \}$ and $\{ E_1 , E_2 , \ldots , E_n \}$ mentioned in Corollary~\ref{c2} are Latin, then the resulting Hamilton decomposition $\{ g(E_i , H_j) \mid 1 \le i \le n \text{ and } 1 \le j \le 3 \} $ is also Latin.
\end{theorem}
\begin{proof}
The proof is very similar to that of Theorem~\ref{t2}, therefore we only sketch it here.
Based on Algorithm~\ref{a3}, if $E_1$ is a source cycle for $Q_{2n}$ with source matrix $M$, and if $H_1$ is a source cycle for $G_{n,3}$ with source matrix
$ \left[ \begin{array}{c|c|c} 1 & 2 & 3 \\ \hline 2 & 3 & 1 \\ \hline 3 & 1 & 2 \end{array} \right]$, then $g(E_1,H_1)$ is a source cycle for $Q_{6n}$ with
$M' = \left[ \begin{array}{c|c|c} M & M+n & M+2n \\ \hline M+n & M+2n & M \\ \hline M+2n & M & M+n \end{array} \right]$ as its source matrix.
\end{proof}
The last theorem gives rise to another important result:
\begin{corollary} \label{c4}
If $Q_{2n}$ has a source cycle, so does $Q_{6n}$.
\end{corollary}
Corollaries \ref{c3} and \ref{c4} give us the main result of this paper:
\begin{corollary} \label{c5}
We have a source cycle for all $Q_{2n}$ with $n = 2^a 3^b$.
\end{corollary}
For future research, we conjecture the following.
\begin{conjecture} \label{conj1}
We have a source cycle for all $Q_{2n}$.
\end{conjecture}

\section*{Acknowledgements}
We thank Negin Karisani for her useful comments during this project.

\bibliographystyle{siam} 
\bibliography{Hamilton_Submit}

\newpage

\appendix

\section{Lemma~\ref{l2} for $n=1$} \label{ap1}

\begin{figure}[H]
\centering
\begin{subfigure}{\textwidth}
  \centering

\scalebox{.8}{
\begin{tikzpicture}

\draw[->,ultra thick] (-2,4)--(-1,4) node[right]{$1$};
\draw[->,ultra thick] (-2,4)--(-2,3) node[below]{$2$};
\draw[->,ultra thick] (-2,4)--(-1.5,3.5) node[below]{$3$};

%\draw[solid,ultra thick] (-3,2)--(-3,1) node[below]{$X$};
%\draw[dashed,ultra thick] (-2.5,2)--(-2.5,1) node[below]{$Y$};
%\draw[densely dotted,thick] (-2,2)--(-2,1) node[below]{$Z$};

\foreach \i in {0,1,2,3}{
	\foreach \j in {0,1,2,3}{
		        \node[vertex, fill, minimum size = 3pt] at (3 * \i, 3 *\j) {};
		        \node[vertex, fill, minimum size = 3pt] at (3 * \i+.5, 3 *\j - .5) {};
      		        \node[vertex, fill, minimum size = 3pt] at (3 * \i+1, 3 *\j - 1) {};
      		        \node[vertex, fill, minimum size = 3pt] at (3 * \i+1.5, 3 *\j - 1.5) {};      		        
	}
}
\node[vertex, fill, minimum size = 6pt] at (0, 9) {};

\draw[solid,ultra thick] (0,9)--(9,9)--(9,6).. controls (4.5,6.2) .. (0, 6)--(6,6)--(6,3)--(9,3) .. controls (4.5,3.2) .. (0, 3)--(3,3)--(3,0)--(9,0) .. controls (4.5,.2) .. (0, 0).. controls (-.2,4.5) .. (0, 9);
%%%%%%%%%%%%%%%%%%%%%%%%%%%%%%%%%%%%%%
\draw[solid,ultra thick] (.5,8.5)--(6.5,8.5)--(6.5,5.5)--(9.5,5.5) .. controls (5,5.7) .. (.5,5.5)--(3.5,5.5)--(3.5,2.5)--(9.5,2.5) .. controls (5,2.7) .. (.5,2.5)-- (.5, -.5)--(9.5, -.5) .. controls (9.3,4) .. (9.5,8.5) .. controls (5,8.7) .. (.5,8.5);
%%%%%%%%%%%%%%%%%%%%%%%%%%%%%%%%%%%%%%
\draw[solid,ultra thick] (1,8)--(4,8)--(4,5)--(10,5) .. controls (5.5,5.2) .. (1,5) -- (1,2) --(10,2) --(10,-1) .. controls (5.5,-.8) .. (1,-1) --(7,-1) .. controls (6.8,3.5) .. (7,8) -- (10,8) .. controls (5.5,8.2) .. (1,8);
%%%%%%%%%%%%%%%%%%%%%%%%%%%%%%%%%%%%%%
\draw[solid,ultra thick](1.5,7.5) -- (1.5,4.5) --(10.5,4.5) --(10.5,1.5) .. controls (6,1.7) .. (1.5,1.5) --(7.5,1.5) -- (7.5,-1.5) -- (10.5,-1.5) .. controls (6,-1.3) .. (1.5,-1.5) --(4.5,-1.5) .. controls (4.3,3) .. (4.5,7.5)--(10.5,7.5) .. controls (6,7.7) .. (1.5,7.5);
\end{tikzpicture}
}
  \subcaption{X \label{X}}
\end{subfigure}
%%%%%%%%%%%%%%%%%%%%%%%
\newline
\begin{subfigure}{.4 \textwidth}
  \centering

\scalebox{.7}{
\begin{tikzpicture}

%\draw[->,ultra thick] (-2,4)--(-1,4) node[right]{$1$};
%\draw[->,ultra thick] (-2,4)--(-2,3) node[below]{$2$};
%\draw[->,ultra thick] (-2,4)--(-1.5,3.5) node[below]{$3$};

%\draw[solid,ultra thick] (-3,2)--(-3,1) node[below]{$X$};
%\draw[dashed,ultra thick] (-2.5,2)--(-2.5,1) node[below]{$Y$};
%\draw[densely dotted,thick] (-2,2)--(-2,1) node[below]{$Z$};

\foreach \i in {0,1,2,3}{
	\foreach \j in {0,1,2,3}{
		        \node[vertex, fill, minimum size = 3pt] at (3 * \i, 3 *\j) {};
		        \node[vertex, fill, minimum size = 3pt] at (3 * \i+.5, 3 *\j - .5) {};
      		        \node[vertex, fill, minimum size = 3pt] at (3 * \i+1, 3 *\j - 1) {};
      		        \node[vertex, fill, minimum size = 3pt] at (3 * \i+1.5, 3 *\j - 1.5) {};      		        
	}
}
\node[vertex, fill, minimum size = 6pt] at (0, 9) {};

\draw[dashed,ultra thick] (0,9)--(0,0)--(.5,-.5) .. controls (.7,4) .. (.5, 8.5)--(.5, 2.5)--(1, 2)--(1,-1) .. controls (1.2,3.5) .. (1, 8)--(1,5)--(1.5,4.5)--(1.5,-1.5) .. controls (1.7,3) .. (1.5, 7.5) .. controls (1,8.5) .. (0, 9);
%%%%%%%%%%%%%%%%%%%%%%%%%%%%%%%%%%%%%%
\draw[dashed,ultra thick] (3,9)--(3,3)--(3.5,2.5)--(3.5,-.5) .. controls (3.7,4) .. (3.5,8.5)--(3.5,5.5)--(4,5)--(4,-1) .. controls (4.2,3.5) .. (4,8) -- (4.5, 7.5)--(4.5, -1.5) .. controls (4,-.5) .. (3,0) .. controls (3.2,4.5) .. (3,9);
%%%%%%%%%%%%%%%%%%%%%%%%%%%%%%%%%%%%%%
\draw[dashed,ultra thick] (6,9)--(6,6)--(6.5,5.5)--(6.5,-.5) .. controls (6.7,4) .. (6.5,8.5) -- (7,8) --(7,-1) --(7.5,-1.5) .. controls (7.7,3) .. (7.5,7.5) --(7.5,1.5) .. controls (7,2.5) .. (6,3) -- (6,0) .. controls (6.2,4.5) .. (6,9);
%%%%%%%%%%%%%%%%%%%%%%%%%%%%%%%%%%%%%%
\draw[dashed,ultra thick](9,9) -- (9.5,8.5) --(9.5,-.5) --(10,-1) .. controls (10.2,3.5) .. (10,8) -- (10,2) -- (10.5,1.5) -- (10.5,-1.5) .. controls (10.7,3) .. (10.5,7.5)--(10.5,4.5) .. controls (10,5.5) .. (9,6)--(9,0) .. controls (9.2,4.5) .. (9,9);
%\draw[dashed,ultra thick](1.5,7.5) -- (1.5,4.5) --(10.5,4.5) --(10.5,1.5) .. controls (6,1.7) .. (1.5,1.5) --(7.5,1.5) -- (7.5,-1.5) -- (10.5,-1.5) .. controls (6,-1.3) .. (1.5,-1.5)--(4.5,-1.5) .. controls (4.3,3) .. (4.5,7.5)--(10.5,7.5) .. controls (6,7.7) .. (1.5,7.5);
\end{tikzpicture}
}

  \subcaption{Y \label{Y}}
\end{subfigure}
%%%%%%%%%%%%%%%%%%%%%%%
%\newline
\hfill
\begin{subfigure}{.4 \textwidth}
  \centering

\scalebox{.7}{
\begin{tikzpicture}

%\draw[->,ultra thick] (-2,4)--(-1,4) node[right]{$1$};
%\draw[->,ultra thick] (-2,4)--(-2,3) node[below]{$2$};
%\draw[->,ultra thick] (-2,4)--(-1.5,3.5) node[below]{$3$};

%\draw[solid,ultra thick] (-3,2)--(-3,1) node[below]{$X$};
%\draw[dashed,ultra thick] (-2.5,2)--(-2.5,1) node[below]{$Y$};
%\draw[densely dotted,thick] (-2,2)--(-2,1) node[below]{$Z$};

\foreach \i in {0,1,2,3}{
	\foreach \j in {0,1,2,3}{
		        \node[vertex, fill, minimum size = 3pt] at (3 * \i, 3 *\j) {};
		        \node[vertex, fill, minimum size = 3pt] at (3 * \i+.5, 3 *\j - .5) {};
      		        \node[vertex, fill, minimum size = 3pt] at (3 * \i+1, 3 *\j - 1) {};
      		        \node[vertex, fill, minimum size = 3pt] at (3 * \i+1.5, 3 *\j - 1.5) {};      		        
	}
}
\node[vertex, fill, minimum size = 6pt] at (0, 9) {};

\draw[densely dotted,thick] (0,9)--(1.5,7.5)--(4.5,7.5) .. controls (4,8.5) .. (3, 9)--(4, 8)--(7,8)--(7.5,7.5) .. controls (7,8.5) .. (6, 9)--(6.5,8.5)--(9.5,8.5)--(10.5,7.5) .. controls (10,8.5) .. (9, 9) .. controls (4.5,9.2) .. (0, 9);
%%%%%%%%%%%%%%%%%%%%%%%%%%%%%%%%%%%%%%
\draw[densely dotted,thick] (0,6)--(1,5)--(4,5)--(4.5,4.5) .. controls (4,5.5) .. (3,6)--(3.5,5.5)--(6.5,5.5)--(7.5,4.5) .. controls (7,5.5) .. (6,6) -- (9, 6)--(10.5, 4.5) .. controls (6,4.7) .. (1.5,4.5) .. controls (1,5.5) .. (0,6);
%%%%%%%%%%%%%%%%%%%%%%%%%%%%%%%%%%%%%%
\draw[densely dotted,thick] (0,3)--(.5,2.5)--(3.5,2.5)--(4.5,1.5) .. controls (4,2.5) .. (3,3) -- (6,3) --(7.5,1.5) --(10.5,1.5) .. controls (10,2.5) .. (9,3) --(10,2) .. controls (5.5,2.2) .. (1,2) -- (1.5,1.5) .. controls (1,2.5) .. (0,3);
%%%%%%%%%%%%%%%%%%%%%%%%%%%%%%%%%%%%%%
\draw[densely dotted,thick](0,0) -- (3,0) --(4.5,-1.5) --(7.5,-1.5) .. controls (7,-.5) .. (6,0) -- (7,-1) -- (10,-1) -- (10.5,-1.5) .. controls (10,-.5) .. (9,0)--(9.5,-.5) .. controls (5,-.3) .. (.5,-.5)--(1.5,-1.5) .. controls (1,-.5) .. (0,0);
%\draw[dashed,ultra thick](1.5,7.5) -- (1.5,4.5) --(10.5,4.5) --(10.5,1.5) .. controls (6,1.7) .. (1.5,1.5) --(7.5,1.5) -- (7.5,-1.5) -- (10.5,-1.5) .. controls (6,-1.3) .. (1.5,-1.5)--(4.5,-1.5) .. controls (4.3,3) .. (4.5,7.5)--(10.5,7.5) .. controls (6,7.7) .. (1.5,7.5);
\end{tikzpicture}
}

  \subcaption{Z \label{Z}}
\end{subfigure}
%%%%%%%%%%%%%%%%%%%%%%%
\caption{ \footnotesize The decomposition discussed in Lemma~\ref{l2} for $n=1$.}\label{pa1}

\end{figure}
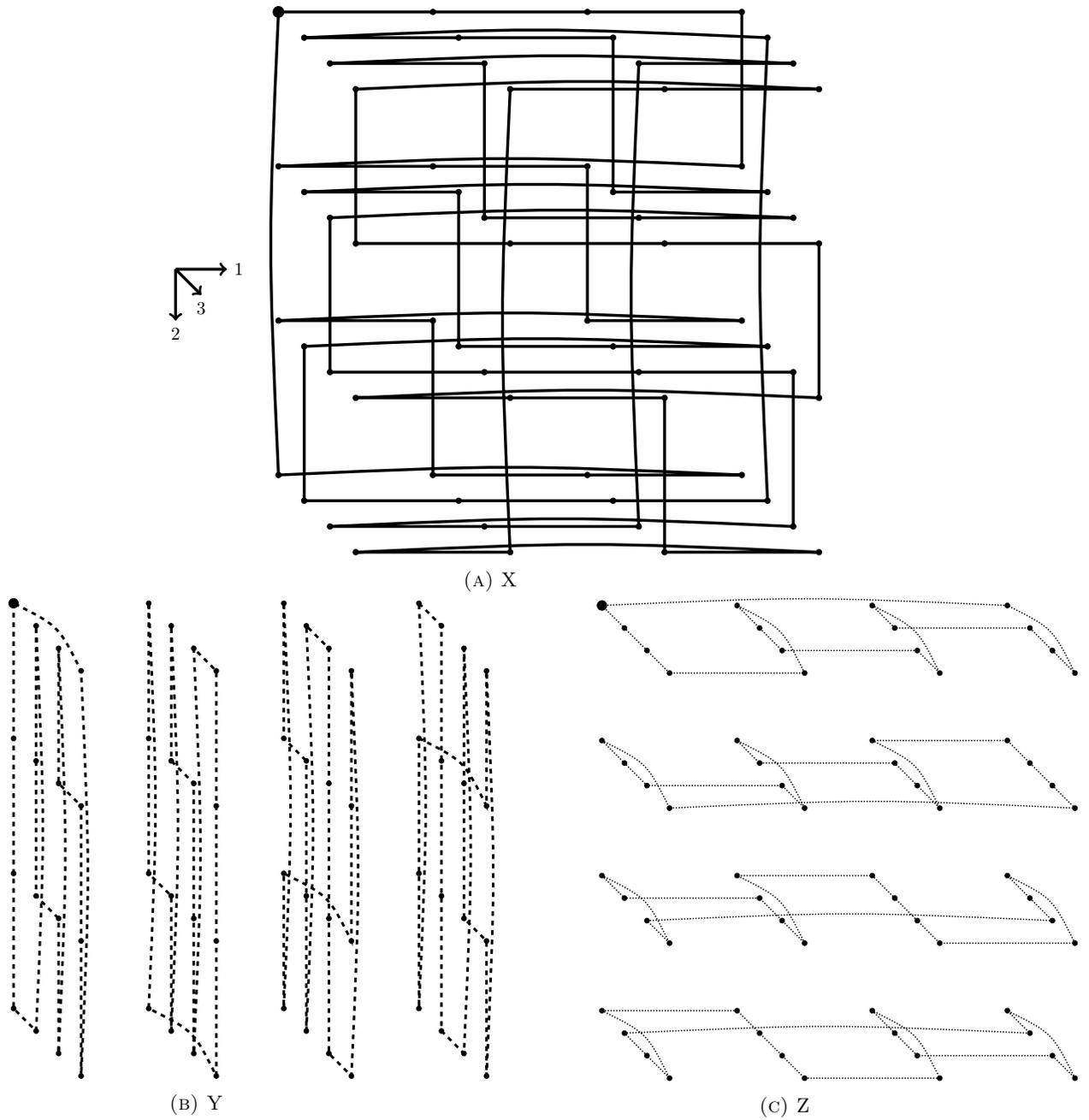

\newpage

\section{Algorithms} \label{ap2}
\subsection{An H.D. for $Q_{4n}$}\hspace*{\fill}  \label{b1}

{\bf Input:}
\begin{itemize}
\item An $n \times 4^{n}$ array $e$ with its {\it $i$th} row showing the {\it $i$th} Hamilton cycle of $Q_{2n}$.
\item A $2 \times 4^{2n}$ array $h$ with its {\it $i$th} row showing the {\it $i$th} Hamilton cycle of $G_{n,2}$.
\end{itemize}
{\bf Output:}
\begin{itemize}
\item A $2n \times 4^{2n}$ array $f$ with its {\it $i$th} row showing the {\it $i$th} Hamilton cycle of $Q_{4n}$.
\end{itemize}

\begin{algorithm}[!htb]
\caption{An H.D. for $Q_{4n}$ from an H.D. for $Q_{2n}$ and an H.D. for $G_{n,2}$}
\label{a1}
\begin{algorithmic}[1]
\For {$i \gets 0 \text{ to } 1$}
    \For {$j \gets 0 \text{ to } 1$}
        \State $c[i][j] \gets 0$ \Comment{initializing the $x$- and $y$-coordinates of the two pointers}
    \EndFor
\EndFor
\For {$j \gets 0 \text{ to } n-1$} \Comment{cycling through $E_1$ to $E_n$}
    \For {$i \gets 0 \text{ to } 4^{2n}-1$} \Comment{cycling through edges of $H_1$ and $H_2$}
        \For {$k \gets 0 \text{ to } 1$} \Comment{cycling through $H_1$ and $H_2$}
            \State $\text{dir} \gets h[k][i][1]$ \Comment{direction of the current edge in $H_{k+1}$} \label{a1l9}
            \State $\text{dim} \gets h[k][i][0] - 1$ \Comment{dimension of the current edge in $H_{k+1}$} \label{a1l10}
            \If {$\text{dir} = 0$ } \Comment{if the current edge in $H_{k+1}$ is forward}
                \State $ f[j+kn][i][0] \leftarrow e[j][c[k][\text{dim}]][0]+n(\text{dim})$ \Comment{\parbox[t]{4.2cm}{dimension of the current edge in $F_{j+1+kn}$}} \label{a1l12}
                \State $ f[j+kn][i][1] \leftarrow e[j][c[k][\text{dim}]][1]$ \Comment{\parbox[t]{4.2cm}{direction of the current edge in $F_{j+1+kn}$}} \label{a1l13}
                \State $ c[k][\text{dim}] \gets c[k][\text{dim}]+1$ \Comment{moving forward in the current copy of $E_{j+1}$}
                \If {$c[k][\text{dim}] = 4^n$ } \Comment{mod operations}
                    \State $c[k][\text{dim}] \leftarrow 0$
                \EndIf
            \Else \Comment{if the current edge in $H_{k+1}$ is backward}
                \State $ c[k][\text{dim}] \gets c[k][\text{dim}]-1$ \Comment{moving backward in the current copy of $E_{j+1}$}
                \If {$c[k][\text{dim}] = -1$ } \Comment{mod operations}
                    \State $c[k][\text{dim}] \leftarrow 4^n-1$
                \EndIf
                \State $ f[j+kn][i][0] \leftarrow e[j][c[k][\text{dim}]][0]+n(\text{dim})$ \Comment{\parbox[t]{4.2cm}{dimension of the current edge in $F_{j+1+kn}$}} \label{a1l23}
                \State $ f[j+kn][i][1] \leftarrow 1 - e[j][c[k][\text{dim}]][1]$ \Comment{\parbox[t]{4.2cm}{direction of the current edge in $F_{j+1+kn}$}} \label{a1l24}
            \EndIf
        \EndFor
    \EndFor
\EndFor
\end{algorithmic}
\end{algorithm}

\newpage
\subsection{An H.D. for $G_{n,3}$}\hspace*{\fill}  \label{b2}

{\bf Input:}
\begin{itemize}
\item A $4^{n} \times 2$ array $S$ having the merging set $S$ in its first $4^n-1$ rows. \\
The elements of $S$ are sorted by their $x$-coordinates, with the {\it$i$th} row of $S$ having the element with $x$-cooridnate $i$. The first entry gives the $y$-coordinate and the second gives the $z$-coordinate.
\end{itemize}
{\bf Output:}
\begin{itemize}
\item A $4^{3n} \times 2$ array $H$ having the edges of $X$.
\end{itemize}

\begin{algorithm}[!htb]
\caption{\small{An algorithm for finding $X$.}}
\label{a2}
\begin{algorithmic}[1]
%\For {$i \gets 0 \text{ to } 2$} \Comment{initializing the pointer}
%	\State $c[i] \gets 0$
%\EndFor
\State $x \gets 0$ \Comment{initializing the pointer's $x$-coordinate}
\State $y \gets 0$ \Comment{initializing the pointer's $y$-coordinate}
\State $z \gets 0$ \Comment{initializing the pointer's $z$-coordinate}
\State $c \gets 0$ \Comment{$c = x+y+z$}
\State \text{dir} $\leftarrow$ 0
\State $s[4^n-1][0] \gets -2$ \Comment{no element of $S$ has $x$-coordinate equal to $4^n-1$}
\State $s[4^n-1][1] \gets -2$
\For {$i \gets 0 \text{ to } 4^n-1$} \Comment{creating the helping sets  $S'$, $D$, $D'$, $M$, and $M'$}
	\State $sp[i][0] \gets s[i][0]$ \Comment{creating the {\it$i$th} member of $S'$}
	\State $sp[i][1] \gets s[i][1] + 1$
	\If {$sp[i][1] = 4^n$} \Comment{mod operations}
		\State $sp[i][1] \gets 0$	
	\EndIf
	\State $d[i][0] \gets s[i][0] + 1$ \Comment{creating the {\it$i$th} member of $D$}
	\State $d[i][1] \gets s[i][1]$
	\If {$d[i][0] = 4^n$} \Comment{mod operations}
		\State $d[i][0] \gets 0$	
	\EndIf	
	\State $dp[i][0] \gets d[i][0]$ \Comment{creating the {\it$i$th} member of $D'$}
	\State $dp[i][1] \gets sp[i][1]$
	\State $m[i+1][0] \gets s[i][0]$ \Comment{creating the {\it$i$th} member of $M$}
	\State $m[i+1][1] \gets sp[i][1]$
	\State $mp[i+1][0] \gets d[i][0]$ \Comment{creating the {\it$i$th} member of $M'$}
	\State $mp[i+1][1] \gets sp[i][1]$
\EndFor
\State $m[0][0] \gets -1$ \Comment{no element of $M$ has $x$-coordinate equal to $0$}
\State $m[0][1] \gets -1$
\State $mp[0][0] \gets -1$ \Comment{no element of $M'$ has $x$-coordinate equal to $0$}
\State $mp[0][1] \gets -1$
\algstore{ss32}
\end{algorithmic}
\end{algorithm}
\newpage
\begin{algorithm}
\begin{algorithmic}
\algrestore{ss32}
\For {$i \gets 0 \text{ to } 4^{3n}-1$} \Comment{main loop for building the {\it$i$th} edge}
\If {$s[x][0] = y$ \textbf{and} $s[x][1] = z$} \Comment{$(x,y,z) \in S$}
	\If {$\text{dir} = 0$} \Comment{we have reached $S$ from outside of cube}
		\State $ \text{dir} \gets 1$ \Comment{in the next step we exit from $S'$ in negative direction}
		\State $h[i][0] \gets 3$  \Comment{the {\it$i$th} edge is in dimension 3}
		\State $h[i][1] \gets 0$  \Comment{the {\it$i$th} edge is in positive direction}
		\State $c \gets c+1$ \Comment{adding $1$ to $z$ and $c$}
		\State $z \leftarrow z+1$
		\If {$z = 4^n$}	\Comment{mod operations}
			\State $z \gets 0$
		\EndIf				
	\Else \Comment{we have reached $S$ from $S'$}
		\State $ \text{dir} \gets 1$ 
		\State $h[i][0] \gets 1$ \Comment{the {\it$i$th} edge is in dimension 1}
		\State $h[i][1] \gets 1$ \Comment{the {\it$i$th} edge is in negative direction}
		\State $c \gets c-1$ \Comment{subtracting $1$ from $x$ and $c$}
		\State $x \leftarrow x-1$
		\If {$x = -1$}	\Comment{mod operations}
			\State $x \gets 4^n-1$
		\EndIf				
	\EndIf
\ElsIf {$sp[x][0] = y$ \textbf{and} $sp[x][1] = z$} \Comment{$(x,y,z) \in S'$}
	\If {$\text{dir} = 0$} \Comment{we have reached $S'$ from outside of cube}
		\State $ \text{dir} \gets 1$ \Comment{in the next step we exit from $S$ in negative direction}
		\State $h[i][0] \gets 3$   \Comment{the {\it$i$th} edge is in dimension 3}
		\State $h[i][1] \gets 1$   \Comment{the {\it$i$th} edge is in negatice direction}
		\State $c \gets c-1$ \Comment{subtracting $1$ from $z$ and $c$}
		\State $z \leftarrow z-1$
		\If {$z = -1$}	\Comment{mod operations}
			\State $z \gets 4^n-1$
		\EndIf				
	\Else \Comment{we have reached $S'$ from $S$}
		\State $ \text{dir} \gets 1$
		\State $h[i][0] \gets 2$ \Comment{the {\it$i$th} edge is in dimension 2}
		\State $h[i][1] \gets 1$ \Comment{the {\it$i$th} edge is in negative direction}
		\State $c \gets c-1$ \Comment{subtracting $1$ from $z$ and $c$}
		\State $y \leftarrow y-1$
		\If {$y = -1$}	\Comment{mod operations}
			\State $y \gets 4^n-1$
		\EndIf				
	\EndIf
\algstore{ss32}
\end{algorithmic}
\end{algorithm}
\newpage
\begin{algorithm}
\begin{algorithmic}
\algrestore{ss32}
\ElsIf {$d[x][0] = y$ \textbf{and} $d[x][1] = z$} \Comment{$(x,y,z) \in D$}
	\If {$\text{dir} = 0$} \Comment{we have reached $D$ from $D'$}
%		\State $ \text{dir} \gets 1$ \Comment{in the next step we exit from $S'$ in negative direction}
		\State $h[i][0] \gets 1$  \Comment{the {\it$i$th} edge is in dimension 1}
		\State $h[i][1] \gets 0$  \Comment{the {\it$i$th} edge is in positive direction}
		\State $c \gets c+1$ \Comment{adding $1$ to $x$ and $c$}
		\State $x \leftarrow x+1$
		\If {$x = 4^n$}	\Comment{mod operations}
			\State $x \gets 0$
		\EndIf				
	\Else \Comment{we have reached $D$ from $V'$}
%		\State $ \text{dir} \gets 1$ 
		\State $h[i][0] \gets 3$ \Comment{the {\it$i$th} edge is in dimension 3}
		\State $h[i][1] \gets 0$ \Comment{the {\it$i$th} edge is in positive direction}
		\State $c \gets c+1$ \Comment{adding $1$ to $x$ and $c$}
		\State $z \leftarrow z+1$
		\If {$z = 4^n$}	\Comment{mod operations}
			\State $z \gets 0$
		\EndIf				
	\EndIf
\ElsIf {$dp[x][0] = y$ \textbf{and} $dp[x][1] = z$} \Comment{$(x,y,z) \in D'$}
	\If {$\text{dir} = 0$} \Comment{we have reached $D'$ from outside of cube}
		%\State $ \text{dir} \gets 1$ \Comment{in the next step we exit from $S$ in negative direction}
		\State $h[i][0] \gets 3$   \Comment{the {\it$i$th} edge is in dimension 3}
		\State $h[i][1] \gets 1$   \Comment{the {\it$i$th} edge is in negative direction}
		\State $c \gets c-1$ \Comment{subtracting $1$ from $z$ and $c$}
		\State $z \leftarrow z-1$
		\If {$z = -1$}	\Comment{mod operations}
			\State $z \gets 4^n-1$
		\EndIf				
	\Else \Comment{we have reached $D'$ from $D$}
		%\State $ \text{dir} \gets 1$
		\State $h[i][0] \gets 1$ \Comment{the {\it$i$th} edge is in dimension 1}
		\State $h[i][1] \gets 1$ \Comment{the {\it$i$th} edge is in negative direction}
		\State $c \gets c-1$ \Comment{subtracting $1$ from $z$ and $c$}
		\State $x \leftarrow x-1$
		\If {$x = -1$}	\Comment{mod operations}
			\State $x \gets 4^n-1$
		\EndIf				
	\EndIf
\algstore{ss33}
\end{algorithmic}
\end{algorithm}
\newpage
\begin{algorithm}
\begin{algorithmic}
\algrestore{ss33}
\ElsIf {$m[x][0] = y$ \textbf{and} $m[x][1] = z$} \Comment{$(x,y,z) \in M$}
    \If {$\text{dir} = 0$} \Comment{we have reached $M$ from $M'$}
%       \State $ \text{dir} \gets 1$ 
        \State $h[i][0] \gets 1$  \Comment{the {\it$i$th} edge is in dimension 1}
        \State $h[i][1] \gets 0$  \Comment{the {\it$i$th} edge is in positive direction}
        \State $c \gets c+1$ \Comment{adding $1$ to $x$ and $c$}
        \State $x \leftarrow x+1$
        \If {$x = 4^n$} \Comment{mod operations}
            \State $x \gets 0$
        \EndIf              
    \Else \Comment{we have reached $M$ from outside of cube}
        \State $ \text{dir} \gets 0$ \Comment{in the next step we exit from $M'$ in positive direction}
        \State $h[i][0] \gets 2$ \Comment{the {\it$i$th} edge is in dimension 3}
        \State $h[i][1] \gets 0$ \Comment{the {\it$i$th} edge is in positive direction}
        \State $c \gets c+1$ \Comment{adding $1$ to $y$ and $c$}
        \State $y \leftarrow y+1$
        \If {$y = 4^n$} \Comment{mod operations}
            \State $y \gets 0$
        \EndIf              
    \EndIf
\ElsIf {$mp[x][0] = y$ \textbf{and} $mp[x][1] = z$} \Comment{$(x,y,z) \in M'$}
    \If {$\text{dir} = 0$} \Comment{we have reached $M'$ from $M$}
        %\State $ \text{dir} \gets 1$ 
        \State $h[i][0] \gets 1$   \Comment{the {\it$i$th} edge is in dimension 1}
        \State $h[i][1] \gets 0$   \Comment{the {\it$i$th} edge is in positive direction}
        \State $c \gets c+1$ \Comment{adding $1$ to $x$ and $c$}
        \State $x \leftarrow x+1$
        \If {$x = 4^n$}  \Comment{mod operations}
            \State $x \gets 0$
        \EndIf              
    \Else \Comment{we have reached $M'$ from outside of cube}
        \State $ \text{dir} \gets 1$ \Comment{in the next step we exit from $M$ in positive direction}
        \State $h[i][0] \gets 2$ \Comment{the {\it$i$th} edge is in dimension 2}
        \State $h[i][1] \gets 1$ \Comment{the {\it$i$th} edge is in negative direction}
        \State $c \gets c-1$ \Comment{subtracting $1$ from $y$ and $c$}
        \State $y \leftarrow y-1$
        \If {$y = -1$}  \Comment{mod operations}
            \State $y \gets 4^n-1$
        \EndIf              
    \EndIf
    \algstore{ss31}
\end{algorithmic}
\end{algorithm}
\newpage
\begin{algorithm}
\begin{algorithmic}
\algrestore{ss31}
\Else \Comment{normal vertex}
    \If {$\text{dir} = 0$} \Comment{if the current direction is positive}
        \If {$c = -1 \pmod{4^n}$} \Comment{if it is time to move in dimension 2}
            \State $h[i][0] \gets 2$ \Comment{the {\it$i$th} edge is in dimension 2}
            \State $h[i][1] \gets 0$ \Comment{the {\it$i$th} edge is in the positive direction}
            \State $c \gets c+1$ \Comment{adding $1$ to $y$ and $c$}
            \State $y \leftarrow y+1$
            \If {$y = 4^n$} \Comment{mod operations}
                \State $y \gets 0$
            \EndIf      
        \Else \Comment{if it is time to move in dimension 1}
            \State $h[i][0] \gets 1$  \Comment{the {\it$i$th} edge is in dimension 1}
            \State $h[i][1] \gets 0$  \Comment{the {\it$i$th} edge is in positive direction}
            \State $c \gets c+1$  \Comment{adding $1$ to $x$ and $c$}
            \State $x \leftarrow x+1$
            \If {$x = 4^n$} \Comment{mod operations}
                \State $x \gets 0$
            \EndIf      
        \EndIf
    \Else \Comment{if the current direction is negative}
        \If {$c = 0 \pmod {4^n}$} \Comment{if it is time to move in dimension 2}
            \State $h[i][0] \gets 2$ \Comment{the {\it$i$th} edge is in dimension 2}
            \State $h[i][1] \gets 1$ \Comment{the {\it$i$th} edge is in negative direction}
            \State $c \gets c-1$  \Comment{subtracting $1$ from $y$ and $c$}
            \State $y \leftarrow y-1$
            \If {$y = -1$}  \Comment{mod operations}
                \State $y \gets 4^n-1$
            \EndIf      
        \Else \Comment{if it is time to move in dimension 1}
            \State $h[i][0] \gets 1$ \Comment{the {\it$i$th} edge is in dimension 1}
            \State $h[i][1] \gets 1$ \Comment{the {\it$i$th} edge is in negative direction}
            \State $c \gets c-1$ \Comment{subtracting $1$ from $x$ and $c$}
            \State $x \leftarrow x-1$
            \If {$x = -1$}  \Comment{mod operations}
                \State $x \gets 4^n-1$
            \EndIf      
        \EndIf      
    \EndIf
\EndIf
\EndFor
\end{algorithmic}
\end{algorithm}

\newpage
\subsection{An H.D. for $Q_{6n}$}\hspace*{\fill} \label{b3}

\textbf{Input:}
\begin{itemize}
\item An $n \times 4^{n}$ array $e$ with its {\it $i$th} row showing the {\it $i$th} Hamilton cycle for $Q_{2n}$.
\item A $3 \times 4^{3n}$ array $h$ with its {\it $i$th} row showing the {\it $i$th} Hamilton cycle for $G_{n,3}$.
\end{itemize}
\textbf{Output:}
\begin{itemize}
\item A $3n \times 4^{3n}$ array $g$ with its {\it $i$th} row showing the {\it $i$th} Hamilton cycle for $Q_{6n}$.
\end{itemize}

\begin{algorithm}[!htb]
\caption{An H.D. for $Q_{6n}$ from an H.D. for $Q_{2n}$ and an H.D. for $G_{n,3}$}
\label{a3}
\begin{algorithmic}[1]
\For {$i \gets 0 \text{ to } 2$}
    \For {$j \gets 0 \text{ to } 2$}
        \State $c[i][j] \gets 0$ \Comment{initializing the $x$-, $y$-, and $z$-coordinates of the three pointers}
    \EndFor
\EndFor
\For {$j \gets 0 \text{ to } n-1$} \Comment{cycling through $E_1$ to $E_n$}
    \For {$i \gets 0 \text{ to } 4^{3n}-1$} \Comment{cycling through edges of $H_1$, $H_2$, and $H_3$}
        \For {$k \gets 0 \text{ to } 2$} \Comment{cycling through $H_1$, $H_2$, and $H_3$}
            \State $\text{dir} \gets h[k][i][1]$ \Comment{direction of the current edge in $H_{k+1}$}
            \State $\text{dim} \gets h[k][i][0] - 1$ \Comment{dimension of the current edge in $H_{k+1}$}
            \If {$\text{dir} = 0$ } \Comment{if the current edge in $H_{k+1}$ is forward}
                \State $ f[j+kn][i][0] \leftarrow e[j][c[k][\text{dim}]][0]+n(\text{dim})$ \Comment{\parbox[t]{4.2cm}{dimension of the current edge in $F_{j+1+kn}$}}
                \State $ f[j+kn][i][1] \leftarrow e[j][c[k][\text{dim}]][1]$ \Comment{\parbox[t]{4.2cm}{direction of the current edge in $F_{j+1+kn}$}}
                \State $ c[k][\text{dim}] \gets c[k][\text{dim}]+1$ \Comment{moving forward in the current copy of $E_{j+1}$}
                \If {$c[k][\text{dim}] = 4^n$ } \Comment{mod operations}
                    \State $c[k][\text{dim}] \leftarrow 0$
                \EndIf
            \Else \Comment{if the current edge in $H_{k+1}$ is backward}
                \State $ c[k][\text{dim}] \gets c[k][\text{dim}]-1$ \Comment{moving backward in the current copy of $E_{j+1}$}
                \If {$c[k][\text{dim}] = -1$ } \Comment{mod operations}
                    \State $c[k][\text{dim}] \leftarrow 4^n-1$
                \EndIf
                \State $ f[j+kn][i][0] \leftarrow e[j][c[k][\text{dim}]][0]+n(\text{dim})$ \Comment{\parbox[t]{4.2cm}{dimension of the current edge in $F_{j+1+kn}$}}
                \State $ f[j+kn][i][1] \leftarrow 1 - e[j][c[k][\text{dim}]][1]$ \Comment{\parbox[t]{4.2cm}{direction of the current edge in $F_{j+1+kn}$}}
            \EndIf
        \EndFor
    \EndFor
\EndFor
\end{algorithmic}
\end{algorithm}

\newpage
\subsection{A source cycle for $Q_{4n}$}\hspace*{\fill} \label{b4}

{\bf Input:}
\begin{itemize}
\item A $4^{n} \times 2$ array $e$ having the source cycle for $Q_{2n}$.
\item An $n$ by $n$ source matrix $A$ for the cycle $E$.
\item A $4^{2n} \times 2$ array $h$ having the source cycle for $G_{n,2}$.
\end{itemize}
{\bf Output:}
\begin{itemize}
\item A $4^{2n} \times 2$ array $f$ having the source cycle for $Q_{4n}$.
\item A $2n$ by $2n$ matrix $P$ as the accompanying source matrix.
\end{itemize}

\begin{algorithm}[!htb]
\caption{A Source Cycle for $Q_{4n}$ From Source Cycles for $Q_{2n}$ and $G_{n,2}$}
\label{a4}
\begin{algorithmic}[1]
\For {$i \gets 0 \text{ to } n-1$} \Comment{building $P$}
    \For {$j \gets 0 \text{ to } n-1$}
    	\For {$k \gets 0 \text{ to } 1$}
	    \For {$t \gets 0 \text{ to } 1$}
	        \State $z \gets (k+t) \pmod{2}$ 
	        \State $p[i+k][j+tn] \gets a[i][j] + zn$ 
	\EndFor
        \EndFor
    \EndFor
\EndFor
\For {$i \gets 0 \text{ to } 1$} \Comment{initializing the $x$- and $y$-coordinates of the pointer}
        \State $c[i] \gets 0$ 
\EndFor
    \For {$i \gets 0 \text{ to } 4^{2n}-1$} \Comment{cycling through edges of $H$}
            \State $\text{dir} \gets h[i][1]$ \Comment{direction of the current edge in $H$}
            \State $\text{dim} \gets h[i][0] - 1$ \Comment{dimension of the current edge in $H$}
            \If {$\text{dir} = 0$ } \Comment{if the current edge in $H$ is forward}
                \State $ f[i][0] \leftarrow e[c[\text{dim}]][0]+n(\text{dim})$ \Comment{dimension of the current edge in $F$}
                \State $ f[i][1] \leftarrow e[c[\text{dim}]][1]$ \Comment{direction of the current edge in $F$}
                \State $ c[\text{dim}] \gets c[\text{dim}]+1$ \Comment{moving forward in the current copy of $E$}
                \If {$c[\text{dim}] = 4^n$ } \Comment{mod operations}
                    \State $c[\text{dim}] \leftarrow 0$
                \EndIf
            \Else \Comment{if the current edge in $H$ is backward}
                \State $ c[\text{dim}] \gets c[\text{dim}]-1$ \Comment{moving backward in the current copy of $E$}
                \If {$c[\text{dim}] = -1$ } \Comment{mod operations}
                    \State $c[\text{dim}] \leftarrow 4^n-1$
                \EndIf
                \State $ f[i][0] \leftarrow e[c[\text{dim}]][0]+n(\text{dim})$ \Comment{dimension of the current edge in $F$}
                \State $ f[i][1] \leftarrow 1 - e[c[\text{dim}]][1]$ \Comment{direction of the current edge in $F$}
            \EndIf
    \EndFor
\end{algorithmic}
\end{algorithm}

\newpage
\subsection{A source cycle for $Q_{6n}$}\hspace*{\fill} \label{b5}

{\bf Input:}
\begin{itemize}
\item A $4^{n} \times 2$ array $e$ having the source cycle for $Q_{2n}$.
\item An $n$ by $n$ source matrix $A$ for the cycle $E$.
\item A $4^{2n} \times 2$ array $h$ having the source cycle for $G_{n,2}$.
\end{itemize}
{\bf Output:}
\begin{itemize}
\item A $4^{3n} \times 2$ array $f$ having the source cycle for $Q_{6n}$.
\item A $3n$ by $3n$ matrix $P$ as the accompanying source matrix.
\end{itemize}

\begin{algorithm}[!htb]
\caption{A Source Cycle for $Q_{6n}$ From Source Cycles for $Q_{2n}$ and $G_{n,3}$} 
\label{a5}
\begin{algorithmic}[1]
\For {$i \gets 0 \text{ to } n-1$} \Comment{building $P$}
    \For {$j \gets 0 \text{ to } n-1$}
    	\For {$k \gets 0 \text{ to } 2$}
	    \For {$t \gets 0 \text{ to } 2$}
	        \State $z \gets (k+t) \pmod{3}$ 
	        \State $p[i+k][j+tn] \gets a[i][j] + zn$ 
	\EndFor
        \EndFor
    \EndFor
\EndFor
\For {$i \gets 0 \text{ to } 2$} \Comment{initializing the $x$-, $y$-, and $z$-coordinates of the pointer}
        \State $c[i] \gets 0$ 
\EndFor
    \For {$i \gets 0 \text{ to } 4^{3n}-1$} \Comment{cycling through edges of $H$}
            \State $\text{dir} \gets h[i][1]$ \Comment{direction of the current edge in $H$}
            \State $\text{dim} \gets h[i][0] - 1$ \Comment{dimension of the current edge in $H$}
            \If {$\text{dir} = 0$ } \Comment{if the current edge in $H$ is forward}
                \State $ f[i][0] \leftarrow e[c[\text{dim}]][0]+n(\text{dim})$ \Comment{dimension of the current edge in $F$}
                \State $ f[i][1] \leftarrow e[c[\text{dim}]][1]$ \Comment{direction of the current edge in $F$}
                \State $ c[\text{dim}] \gets c[\text{dim}]+1$ \Comment{moving forward in the current copy of $E$}
                \If {$c[\text{dim}] = 4^n$ } \Comment{mod operations}
                    \State $c[\text{dim}] \leftarrow 0$
                \EndIf
            \Else \Comment{if the current edge in $H$ is backward}
                \State $ c[\text{dim}] \gets c[\text{dim}]-1$ \Comment{moving backward in the current copy of $E$}
                \If {$c[\text{dim}] = -1$ } \Comment{mod operations}
                    \State $c[\text{dim}] \leftarrow 4^n-1$
                \EndIf
                \State $ f[i][0] \leftarrow e[c[\text{dim}]][0]+n(\text{dim})$ \Comment{dimension of the current edge in $F$}
                \State $ f[i][1] \leftarrow 1 - e[c[\text{dim}]][1]$ \Comment{direction of the current edge in $F$}
            \EndIf
    \EndFor
\end{algorithmic}
\end{algorithm}

\newpage

\section{Correctness of Algorithm \ref{a1}} \label{ap3}

For a fixed $j$, in the $i$-loop, Algorithm \ref{a1} outputs two cycles $F_{j+1} = f(E_{j+1}, H_1)$ and $F_{j+1+n} = f(E_{j+1}, H_2)$.
It does so by traversing the edges of $H_1$ ($k=0$) and $H_2$ ($k=1$) and mimicking them:
\begin{itemize}
    \item If the {\it$i$th} edge of $H_1$ is $1$ or $\overline{1}$, then the {\it$i$th} edge of $F_{j+1}$ is one of $\left\{ 1, \overline{1}, 2, \overline{2}, \ldots, n, \overline{n} \right\}$, and if the {\it$i$th} edge of $H_1$ is $2$ or $\overline{2}$, then the {\it$i$th} edge of $F_{j+1}$ is one of
    \\$\left\{ n+1, \overline{n+1}, n+2, \overline{n+2}, \ldots, 2n, \overline{2n} \right\}$.
    Same thing is true for $H_2$ and $F_{j+1+n}$.
    \item A pointer, with its $x$- and $y$-coordinates being $c[0][0]$ and $c[0][1]$, tracks movement through $H_1$ on the 2-dimensional grid.
    Another pointer, with its $x$- and $y$-coordinates being $c[1][0]$ and $c[1][1]$, tracks movement through $H_2$ on the 2-dimensional grid.
    These pointers together with $E_{j+1}$ determine in what dimension and direction the {\it$i$th} edges of $F_{j+1}$ $F_{j+1+n}$ are:
    \begin{itemize}
        \item If the {\it$i$th} edge of $H_1$ is from $(a,b)$ to $(a+1,b)$, then the {\it$i$th} edge of $F_{j+1}$ has the same direction and dimension as the {\it$(a+1)$st} edge of $E_{j+1}$.
        \item If the {\it$i$th} edge of $H_1$ is from $(a,b)$ to $(a,b+1)$, then the {\it$i$th} edge of $F_{j+1}$ has direction equal to that of the {\it$(b+1)$st} edge of $E_{j+1}$ and dimension equal to $n$ plus the dimension of the {\it$(b+1)$st} edge of $E_{j+1}$.
        \item If the {\it$i$th} edge of $H_1$ is from $(a,b)$ to $(a-1,b)$, then the {\it$i$th} edge of $F_{j+1}$ has direction opposite to that of the {\it$a$th} edge of $E_{j+1}$ and dimension equal to that of the {\it$a$th} edge of $E_{j+1}$.
        \item If the {\it$i$th} edge of $H_1$ is from $(a,b)$ to $(a,b-1)$, then the {\it$i$th} edge of $F_{j+1}$ has direction opposite to that of the {\it$b$th} edge of $E_{j+1}$ and dimension equal to $n$ plus the dimension of the {\it$b$th} edge of $E_{j+1}$.
        \item Same things can be said about $H_2$ and $E_{j+1}$.
    \end{itemize}
    \item The pointers are initially set to $(0,0)$.
    After each iteration of $j$, the pointers become $(0,0)$ because $H_1$ and $H_2$ start and end at $(0,0)$.
\end{itemize}

\end{document}